\documentclass{article}
\usepackage{graphicx}

\usepackage{amsmath}
\usepackage{amssymb}
\usepackage{comment}
\usepackage[noadjust]{cite}
\usepackage{enumerate}
\usepackage{mathrsfs}
\usepackage{mathtools}
\usepackage{subcaption}
\usepackage{siunitx}
\usepackage{proof}
\usepackage{qtree}
\usepackage{url}
\usepackage{xcolor}
\definecolor{1f1e33}{HTML}{1F1E33}
\definecolor{mediumblue}{HTML}{0000CD}

\usepackage{tikz}
\usetikzlibrary{automata, arrows.meta, positioning}

\usepackage[all]{xy}

\usepackage{amsthm}
\theoremstyle{definition}
\newtheorem{dfn}{Definition}[section]
\newtheorem{prop}[dfn]{Proposition}
\newtheorem{lem}[dfn]{Lemma}
\newtheorem{thm}[dfn]{Theorem}
\newtheorem{cor}[dfn]{Corollary}
\newtheorem{rem}[dfn]{Remark}

\newtheorem{ex}[dfn]{Example}

\newcommand{\Prop}{\mathbf{Prop}}
\newcommand{\Nom}{\mathbf{Nom}}
\newcommand{\dia}{\diamondsuit}
\newcommand{\Kmodel}{\mathfrak{M}}
\newcommand{\Kframe}{\mathfrak{F}}

\usepackage{listings}
\usepackage[colorlinks=true,allcolors=1f1e33]{hyperref}
\usepackage{nameref}

\title{Completeness of Tableau Calculi for Two-Dimensional Hybrid Logics}

\author{Yuki Nishimura}
\date{}
\begin{document}
\maketitle

\begin{abstract}
Hybrid logic is one of the extensions of modal logic. The many-dimensional product of hybrid logic is called hybrid product logic (HPL). We construct a sound and complete tableau calculus for two-dimensional HPL. Also, we made a tableau calculus for hybrid dependent product logic (HdPL), where one dimension depends on the other. In addition, we add a special rule to the tableau calculus for HdPL and show that it is still sound and complete. All of them lack termination, however.
\end{abstract}

\section{Introduction}
\label{secintro}

Hybrid logic is an extension of basic modal logic with additional propositional symbols called nominals and an operator called a satisfaction operator $@$. A nominal is a formula that is true only in a single possible world in a Kripke model. Simply put, we can assume that it is a syntactic name of a possible world. Also, with the satisfaction operator, we can create a formula of the form $@_i p$, which means ``the proposition $p$ holds at point $i$.'' Hybrid logic was first invented by Prior \cite{prior1967, prior1968} and has evolved in various ways since then (see \cite{tencate2005, blackburn2006P, indrzejczak2007, brauner2011}).

For example, consider the following sentence: ``Alice publishes a book on September 1.'' Using hybrid language, we can write it as $@_i p$, where $p =$ ``Alice will publish a book'' and $i =$ ``It is September 1.'' In this way, we can describe one particular point, such as a moment, a place, or a person.

Needless to say, $\square$ and $\dia$ are familiar from ordinary languages of modal logic (e.g., \cite{blackburn2002}). On the other hand, nominals can be said to incorporate the idea of ``individual constants'' in predicate logic. In this respect, hybrid logic is a ``hybrid'' of modal logic and predicate logic. Furthermore, we can also reinterpret the word ``hybrid'' from a different perspective. In modal logic, $\square \varphi$ and $\dia \varphi$ are local descriptions from the viewpoint of where we are now \cite[Preface]{blackburn2002}. On the other hand, satisfaction operators, which point a single world, enable us to write a global description from a bird's-eye view of the model. In this respect, hybrid logic is a hybrid of local and global descriptions \cite{blackburn1995, blackburn2006A}.

The advantage of hybrid logic is that, when considering modal logic as a ``logic of relational structures'' \cite[Preface]{blackburn2002}, hybrid logic is more expressive than modal logic in terms of definability of relational structures, i.e., frame definability. For example, in basic modal logic, there is no axiom corresponding to the irreflexivity of the model \cite{goldblatt1987}. In hybrid logic, however, we have a simple axiom, $@_i \neg \dia i$, where $i$ is a nominal \cite{blackburn1999}. 

The expansion of hybrid logic to multiple dimensions is called many-dimensional hybrid logic. To illustrate what it is, let us examine one example of two-dimesional hybrid logic, borrowing an idea from \cite{sano2010}. Consider the following scenario:

\begin{quotation}
  It is 10 o'clock, and you are writing a paper in your room on the 1st floor. Then, you get an e-mail: we have an urgent matter to discuss, so please meet in room 1001 on the 10th floor at 12 o'clock. You think 12 o'clock is in the future, and room 1001 is above your room. Thus, the emergency meeting will be held upstairs in the future.
\end{quotation}

This deduction can be formalized as follows:
\[
  @_i @_a p \land \langle \textsf{Future} \rangle i \land \langle \textsf{Up} \rangle a \rightarrow \langle \textsf{Future} \rangle \langle \textsf{Up} \rangle p ,
\]
where $p =$ ``A meeting is held,'' $i =$ ``It is 12 o'clock,'' and $a =$ ``It is on 10th floor.'' In this example, we regard one dimension as time and the other as place. The propositions ``It is 12 o'clock'' and ``It is on the 10th floor'' are treated as nominals since they point to one moment and one area, respectively.

The formula $@_i @_a p$ in the example above, combining two kinds of satisfaction operators, describes some properties of a unique world. In this point of view, many-dimensional hybrid logics inherit the globality of orthodox (i.e., one-dimensional) hybrid logic.

However, note that there is a case in which we cannot keep the globality of hybrid logic. In the example above, nominal $i$ points to the worlds that are at 12 o'clock, so $@_i p$ holds only at 12 o'clock. In this case, it is a global expression about the dimension of time. However, this proposition says nothing about where you are. So the proposition $@_i p$ holds on the 10th floor, but it does not hold in your room on the 1st floor, which loses the globality.

Another interesting example of two-dimesional hybrid logic is Facebook logic \cite{seligman2011}, which multiplies two concepts, personal friendships and knowledge, and describes the influence of each other. For instance, consider this sentence: ``I am a friend of Andy, and Andy knows that Alice published a book. Then, one of my friends knows that Alice had published a book.'' This inference can be written using the language of Facebook logic, as follows:

\[
  \langle \textsf{Friend} \rangle i \land @_i [\textsf{Know}] p \rightarrow \langle \textsf{Friend} \rangle [\textsf{Know}] p ,
\]
where $p =$ ``Alice published a book'' and $i =$ ``This is Andy.'' 

Facebook logic makes the most of the strong points of hybrid logic. Friendship has the property that no one is a friend of themself, and the corresponding property of the binary relation is irreflexivity. As mentioned earlier, this property can be expressed not in the modal language but in the hybrid language. Furthermore, mapping a nominal of the hybrid language to an agent makes it possible to express friendships of specific agents. For example, the formula $@_i \langle \textsf{Friend} \rangle j$ has the meaning that ``$i$ has a friend satisfying $j$.'' Since the proposition $j$ holds for only one agent, i.e., $j$ is the agent's name, we can interpret this formula as meaning that ``$i$ is a friend of $j$.''

In particular, when all the dimensions are independent of each other, it is called hybrid product logic (HPL). For two-dimensional HPL, a Hilbert-style axiomatization was made in \cite{sano2010}. Then, the sequent calculus of HPL was proposed in \cite{sano2015}. However, there is no tableau calculus for many-dimensional hybrid logic. 

Tableau calculi for orthodox hybrid logic have been studied extensively in \cite{bolander2007} (see also \cite{bolander2009, indrzejczak2013, blackburn2017}). In the case of two-dimensional modal logic, tableau calculi typically use labeled formulae of the form $(x, y): \varphi$ instead of ordinary formulae. However, the internalized tableau calculus we propose here treats the labels as part of the hybrid formulae (we write $@_i @_a \varphi$ as the meaning of $(i, a): \varphi$), allowing us to create tableaux in which only non-labeled formulae appear. As this paper will show, this seems to be a natural way to explore HPL, hybrid dependent product logic (HdPL), and their extensions.

In this paper, we have constructed a sound and complete tableau calculus for two-dimensional HPL and also for HdPL, where one dimension depends on the other, based on the method of \cite{bolander2007}. It lacks termination, however, so the decidability of HPL is still an open problem.  

One of the main differences to \cite{sano2015} is how to treat completeness and cut-elimination. In \cite{sano2015}, the completeness of the sequent calculus was indirectly verified by showing that it is equivalent to the Hilbert system. Moreover, in this paper, the cut-elimination theorem is shown without proof, perhaps because the proof is too complex to write briefly. In contrast, we introduce the cut-free tableau calculi, and their completeness is directly demonstrated by constructing a model that satisfies the root formula from an arbitrary tableau. These properties are helpful since if a formula $\varphi$ is not provable, then we can make a counterexample model of $\varphi$. 

Besides, this tableau calculus and Sano's sequent calculus differ for some points. For example, we directly create the proof system of HPL. This is different from \cite{sano2015} since it constructs two-dimensional hybrid logic whose dimensions depend on each other, and it treats HPL as an application. Another more technical difference is that our rules (specifically $[\mathit{Id}_n]$ and $[\mathit{Id'}_n]$) only allow us to change the ``label'' of nominals, though Equality Rules in \cite{sano2015} let us replace nominals in an arbitrary occurrence with equivalent ones.

We proceed as follows: Section \ref{secsemantics} introduces the semantics of HPL according to \cite{sano2010}, and Section \ref{sectableau} introduces the tableau calculus of HPL. Then, we show soundness and completeness in Sections \ref{secsound} and \ref{seccomplete}, respectively. Furthermore, in Section \ref{secHdPL}, we introduce the semantics and the tableau calculus of HdPL, a logic with dependencies on dimensions, and show that soundness and completeness hold. In particular, in Section \ref{subsecinc}, we introduce the property ``decreasing'' of the HdPL model and show that it can be successfully incorporated into the tableau calculus of HdPL.

\section{Kripke Semantics of Hybrid Product Logic}
\label{secsemantics}
\begin{dfn}
  The vocabulary of HPL is the following:
  \begin{enumerate}[i)]
    \item a countable set $\Prop = \{ p, q, r, \dots \}$ of propositional variables,
    \item two disjoint countable sets $\Nom_1 = \{ i, j, k, \dots \}$ and $\Nom_2 = \{ a, b, c, \dots \}$ of nominals, which are disjoint from $\Prop$,
    \item two boolean connectives: $\neg$ and $\land$,
    \item two modal operators: $\diamondsuit_1$ and $\diamondsuit_2$,
    \item two kinds of satisfaction operators: $@_i (i \in \Nom_1)$ and $@_a (a \in \Nom_2)$.
  \end{enumerate}
  We denote the set of all nominals by $\Nom$, i.e., $\Nom = \Nom_1 \cup \Nom_2$.
\end{dfn}

\begin{dfn}
  The set of formulae of HPL $\mathcal{L}_{\textbf{HPL}}$ is defined inductively as follows:
  \begin{align*}
    \varphi \Coloneqq p \mid i \mid a \mid \neg \varphi \mid \varphi \land \varphi \mid \diamondsuit_1 \varphi \mid \diamondsuit_2 \varphi \mid @_i \varphi \mid @_a \varphi,
  \end{align*}
  where $p \in \Prop, i \in \Nom_1,$ and $a \in \Nom_2$.

  We define the following abbreviations:
  \begin{align*}
    \varphi \lor \psi \coloneqq \neg(\neg \varphi \land \neg \psi), \ \varphi \rightarrow \psi \coloneqq \neg(\varphi \land \neg \psi), \ \square_k \varphi \coloneqq \neg \diamondsuit_k \neg \varphi \ (k = 1, 2).
  \end{align*}
  Also, we use the following abbreviations for a finite set of formulae $\Gamma = \{ \gamma_1, \ldots, \gamma_n \}$:
  \begin{align*}
    \bigwedge \Gamma &\coloneqq \gamma_1 \land \cdots \land \gamma_n, \\
    \bigvee \Gamma &\coloneqq \gamma_1 \lor \cdots \lor \gamma_n. 
  \end{align*}
\end{dfn}

Now, we are ready to provide semantics. Before showing the definition of HPL models, let us review the Kripke frame of basic (one-dimensional) modal logic.

\begin{dfn}
  A \emph{Kripke frame} $\mathfrak{F} = (W, R)$ is defined as follows:
  \begin{itemize}
    \item $W$ is a non-empty countable set.
    \item $R$ is a binary relation on $W$. That is, $R \subseteq W \times W$.
  \end{itemize}
  We call an element of $W$ a \emph{possible world}, and $R$ the \emph{accessibility relation}.
\end{dfn}

For any binary relation $X$ on a set, we write $x X y$ to mean $(x,y) \in X$.

Now we introduce the product Kripke model, the semantics of HPL.

\begin{dfn}
  Given two Kripke frames $\mathfrak{F}_1 = (W_1, R_1)$ and $\mathfrak{F}_2 = (W_2, R_2)$, the \emph{product Kripke frame (product frame)} $\mathfrak{F} = (W, R_h, R_v)$ is defined as follows:
  \begin{itemize}
    \item $W = W_1 \times W_2$.
    \item $(x, y) R_h (x', y')$ iff $x R_1 x'$ and $y = y'$.
    \item $(x, y) R_v (x', y')$ iff $x = x'$ and $y R_2 y'$.
  \end{itemize}

  Furthermore, we define the \emph{product Kripke model (product model)} $\mathfrak{M} = (\mathfrak{F}, V)$ as follows:
  \begin{itemize}
    \item $\mathfrak{F}$ is a product frame.
    \item $V: \Prop \cup \Nom \rightarrow \mathcal{P}(W)$, where
    \begin{itemize}
      \item If $i \in \Nom_1$, then $V(i) = \{ x \} \times W_2 \ \text{for some} \ x \in W_1$.
      \item If $a \in \Nom_2$, then $V(a) = W_1 \times \{ y \} \ \text{for some} \ y \in W_2$.
    \end{itemize}
  \end{itemize}

  We call $V$ the \emph{valuation function}.
  \label{defmodel}
\end{dfn}

The subscripts $h$ and $v$ of $R_h$ and $R_v$ indicate ``horizontal'' and ``vertical.'' Figure \ref{figframe} is an example of a product frame. We often denote a product frame as if we write Cartesian coordinates. 

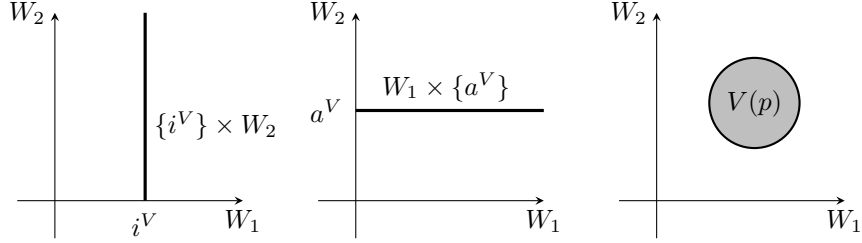
\begin{figure}[tb]
  \centering
  \begin{tikzpicture}
    \draw[->,>=stealth](-0.5,0)--(2.5,0)node[below]{$W_1$};
    \draw[->,>=stealth](0,-0.5)--(0,2.5)node[left]{$W_2$};
    \draw[very thick] (1.2,0)--(1.2,2.5);
    \draw (1.20,1)node[right]{$\{ i^V \} \times W_2$};
    \draw (1.2,-0.05)node[below]{$i^V$};
    \draw[->,>=stealth](3.5,0)--(6.5,0)node[below]{$W_1$};
    \draw[->,>=stealth](4,-0.5)--(4,2.5)node[left]{$W_2$};
    \draw[very thick] (4,1.2)--(6.5,1.2);
    \draw (5.2,1.20)node[above]{$W_1 \times \{ a^V \}$};
    \draw (3.95,1.2)node[left]{$a^V$};
    \draw[->,>=stealth](7.5,0)--(10.5,0)node[below]{$W_1$};
    \draw[->,>=stealth](8,-0.5)--(8,2.5)node[left]{$W_2$};
    \filldraw[fill=lightgray,thick](9.3,1.3) circle [x radius=0.6, y radius=0.6, rotate=0];
    \draw (9.3,1.3)node[]{$V(p)$};
  \end{tikzpicture}
  \caption{A product Kripke model.}
  \label{figframe}
\end{figure}

We abbreviate $x \in W_1$ such that $V(i) = \{ x \} \times W_2$ by $i^V$. We do the same for $a^V$.

\begin{dfn}
  Given a product model $\mathfrak{M}$, a pair of possible worlds $(x, y)$ in $\mathfrak{M}$, and a formula $\varphi$, the \emph{satisfaction relation} $\mathfrak{M}, (x, y) \models \varphi$ is defined inductively as follows:
  \begin{align*}
    \mathfrak{M}, (x, y) \models p \quad &\text{iff} \quad (x, y) \in V(p), \text{where} \ p \in \Prop, \\
    \mathfrak{M}, (x, y) \models i \quad &\text{iff} \quad x = i^V, \text{where} \ i \in \Nom_1, \\
    \mathfrak{M}, (x, y) \models a \quad &\text{iff} \quad y = a^V, \text{where} \ a \in \Nom_2, \\
    \mathfrak{M}, (x, y) \models \neg \varphi \quad &\text{iff} \quad \text{not} \ \mathfrak{M}, (x, y) \models \varphi, \\
    \mathfrak{M}, (x, y) \models \varphi \land \psi \quad &\text{iff} \quad \mathfrak{M}, (x, y) \models \varphi \ \text{and} \ \mathfrak{M}, (x, y) \models \psi, \\
    \mathfrak{M}, (x, y) \models \diamondsuit_1 \varphi \quad &\text{iff} \quad \text{there is some} \ (x', y') \ \text{s.t.} \ (x, y)R_h(x', y') \ \text{and} \ \mathfrak{M}, (x', y') \models \varphi, \\
    \mathfrak{M}, (x, y) \models \diamondsuit_2 \varphi \quad &\text{iff} \quad \text{there is some} \ (x', y') \ \text{s.t.} \ (x, y)R_v(x', y') \ \text{and} \ \mathfrak{M}, (x', y') \models \varphi, \\
    \mathfrak{M}, (x, y) \models @_i \varphi \quad &\text{iff} \quad \mathfrak{M}, (i^V, y) \models \varphi, \\
    \mathfrak{M}, (x, y) \models @_a \varphi \quad &\text{iff} \quad \mathfrak{M}, (x, a^V) \models \varphi.
  \end{align*}

  A formula $\varphi$ is \emph{valid on an $\mathfrak{M}$}, denoted by $\mathfrak{M} \models \varphi$, if $\mathfrak{M}, (x, y) \models \varphi$ holds for all $(x, y) \in W $. Also, if $\mathfrak{M} \models \varphi$ holds for all models $\mathfrak{M}$, we say that $\varphi$ is \emph{valid}, and we write $\models \varphi$.

  Moreover, for two finite sets of formulae $\Gamma$ and $\Delta$, we write $\Gamma \models \Delta$ if $\models \bigwedge \Gamma \rightarrow \bigvee \Delta$ holds.
  \label{defsatis}
\end{dfn}

Intuitively, $\diamondsuit_1 \varphi$ ($\dia_2 \varphi$) represents ``$\varphi$ holds somewhere in the world reachable from here by $R_h$ ($R_v$).'' And $@_i @_a \varphi$ means ``$\varphi$ holds at the point where $i$ and $a$ hold'' or ``$\varphi$ holds at the point $(i,a)$.''

By the definition of $R_h$ and $R_v$, we can rewrite the definition of the satisfaction relation for a formula that has the form $\diamondsuit_k \varphi (k=1,2)$:

\begin{align*}
  \mathfrak{M}, (x, y) \models \diamondsuit_1 \varphi \quad &\text{iff} \quad \text{there is some} \ x' \in W_1 \ \text{s.t.} \ xR_1x' \ \text{and} \ \mathfrak{M}, (x', y) \models \varphi, \\
  \mathfrak{M}, (x, y) \models \diamondsuit_2 \varphi \quad &\text{iff} \quad \text{there is some} \ y' \in W_2 \ \text{s.t.} \ yR_2y' \ \text{and} \ \mathfrak{M}, (x, y') \models \varphi. \\
\end{align*}

We can derive the following satisfaction relation for $\lor , \rightarrow , \square_1,$ and $\square_2$:
\begin{align*}
  \mathfrak{M}, (x, y) \models \varphi \lor \psi \quad &\text{iff} \quad \mathfrak{M}, (x, y) \models \varphi \ \text{or} \ \mathfrak{M}, (x, y) \models \psi, \\
  \mathfrak{M}, (x, y) \models \varphi \rightarrow \psi \quad &\text{iff} \quad \text{if} \ \mathfrak{M}, (x, y) \models \varphi , \ \text{then} \ \mathfrak{M}, (x, y) \models \psi, \\
  \mathfrak{M}, (x, y) \models \square_1 \varphi \quad &\text{iff} \quad \text{for all} \ (x', y') \ , \text{if} \ (x, y)R_h(x', y') , \ \text{then} \ \mathfrak{M}, (x', y') \models \varphi \\
  &\text{iff} \quad \text{for all} \ x' \in W_1, \ \text{if} \ xR_1x' , \ \text{then} \ \mathfrak{M}, (x', y) \models \varphi, \\
  \mathfrak{M}, (x, y) \models \square_2 \varphi \quad &\text{iff} \quad \text{for all} \ (x', y') \ , \text{if} \ (x, y)R_v(x', y') , \ \text{then} \ \mathfrak{M}, (x', y') \models \varphi \\
  &\text{iff} \quad \text{for all} \ y' \in W_2, \ \text{if} \ yR_2y' , \ \text{then} \ \mathfrak{M}, (x, y') \models \varphi. \\
\end{align*}

\section{Internalized Tableau}
\label{sectableau}

For modal logic, the most established tableau calculus is a prefixed tableau calculus\cite{priest2008}. In this system, each formula occurs with a prefix representing a possible world. In hybrid logic, however, we can remove all prefixes. First, an expression of the form ``$i: \varphi$'' that appears in the prefix tableau can be written as $@_i \varphi$. Also, the notation $irj$, which shows reachability, can be written as $@_i \dia j$. Finally, all symbolic sequences that appear in a tableau can be formulae. Such a system is called internalized tableau calculus.

Now, we will construct a tableau calculus for HPL based on an idea proposed in \cite{bolander2007}.

\begin{dfn}
  A \emph{tableau} is a well-founded tree constructed in the following way:
  \begin{itemize}
    \item Start with a formula of the form $@_i @_a \neg \varphi$, where $\varphi$ is a formula and both $i \in \Nom_1$ and $a \in \Nom_2$ do not occur in $\varphi$ (this is called the \emph{root formula}).
    \item For each branch, extend it by applying rules (see Definition \ref{defrules}) to all nodes as often as possible. However, we can no longer add any formula in a branch if at least one of the following conditions is satisfied:
    \begin{enumerate}[i)]
      \item Every new formula generated by applying an arbitrary rule already exists in the branch.
      \item The branch is closed (see Definition \ref{defclose}).
    \end{enumerate}
  \end{itemize}
  Here, a \emph{branch} means a maximal path of the tableau. If a formula $\varphi$ occurs in a branch $\Theta$, then we write $\varphi \in \Theta$.
  \label{deftableau}
\end{dfn}

\begin{dfn}
  A branch of a tableau $\Theta$ is \emph{closed} if $\Theta$ satisfies at least one of the following three conditions:
  \begin{enumerate}[i)]
    \item There is some $\varphi$, $i \in \Nom_1$ and $a \in \Nom_2$ such that $@_i @_a \varphi, @_i @_a \neg \varphi \in \Theta$.
    \item There is some $\varphi$ and $i \in \Nom_1$ such that $@_i \varphi, @_i \neg \varphi \in \Theta$.
    \item There is some $\varphi$ and $a \in \Nom_2$ such that $@_a \varphi, @_a \neg \varphi \in \Theta$.
  \end{enumerate}

  We say that $\Theta$ is \emph{open} if it is not closed. A tableau is called \emph{closed} if all branches in the tableau are closed.
  \label{defclose}
\end{dfn}

\begin{dfn}
  We provide the rules of the tableau calculus for HPL in Figure \ref{figrules}.
  \begin{figure}[t]
  \begin{gather*}
    \infer[{[\neg \neg]}]
      {@_i @_a \varphi}
      {@_i @_a \neg \neg \varphi}
    \qquad
    \infer[{[\land]}]
      {\deduce{@_i @_a \psi} {@_i @_a \varphi}}
      {@_i @_a (\varphi \land \psi)}
    \qquad
    \infer[{[\neg \land]}]
      {@_i @_a \neg \varphi \mid @_i @_a \neg \psi}
      {@_i @_a \neg (\varphi \land \psi)} \\
    \\
    \infer[{[\diamondsuit_1]}^{*1, *3}]
      {\deduce{@_j @_a \varphi}{@_i \diamondsuit_1 j}}
      {@_i @_a \diamondsuit_1 \varphi}
    \qquad
    \infer[{[\diamondsuit_2]}^{*2, *3}]
      {\deduce{@_i @_b \varphi}{@_a \diamondsuit_2 b}}
      {@_i @_a \diamondsuit_2 \varphi}
    \qquad
    \infer[{[\neg \diamondsuit_1]}]
      {@_j @_a \neg \varphi}
      {\deduce{@_i \diamondsuit_1 j}{@_i @_a \neg \diamondsuit_1 \varphi}}
    \qquad
    \infer[{[\neg \diamondsuit_2]}]
      {@_i @_b \neg \varphi}
      {\deduce{@_a \diamondsuit_2 b}{@_i @_a \neg \diamondsuit_2 \varphi}} \\
    \\
    \infer[{[@_1]}]
      {@_j @_a \varphi}
      {@_i @_a @_j \varphi}
    \qquad
    \infer[{[@_2]}]
      {@_i @_b \varphi}
      {@_i @_a @_b \varphi}
    \qquad
    \infer[{[\neg @_1]}]
      {@_j @_a \neg \varphi}
      {@_i @_a \neg @_j \varphi}
    \qquad
    \infer[{[\neg @_2]}]
      {@_i @_b \neg \varphi}
      {@_i @_a \neg @_b \varphi} \\
    \\
    \infer[{[Red_1]}^{*4}]
      {@_i s}
      {@_i @_a s}
    \qquad
    \infer[{[Red_2]}^{*5}]
      {@_a t}
      {@_i @_a t}
    \qquad
    \infer[{[\neg_1]}]
      {@_j j}
      {@_i \neg j}
    \qquad
    \infer[{[\neg_2]}]
      {@_b b}
      {@_a \neg b} \\
    \\
    \infer[{[Id_{1}]}]
      {@_j @_a \varphi}
      {\deduce{@_i j} {@_i @_a \varphi}}
    \qquad
    \infer[{[Id_{2}]}]
      {@_i @_b \varphi}
      {\deduce{@_a b} {@_i @_a \varphi}}
    \qquad
    \infer[{[Id'_{1}]}^{*4}]
      {@_j s}
      {\deduce{@_i j} {@_i s}}
    \qquad
    \infer[{[Id'_{2}]}^{*5}]
      {@_b t}
      {\deduce{@_a b} {@_a t}}
    \qquad
  \end{gather*}

  *1: $j \in \Nom_1$ does not occur in the branch.

  *2: $b \in \Nom_2$ does not occur in the branch.

  *3: This rule can be applied only one time per formula.

  *4: $s = k, \neg k \ (k \in \Nom_1)$.

  *5: $t = c, \neg c \ (c \in \Nom_2)$.

  In these rules, the formulae above the line show the formulae that have already occurred in the branch, and the formulae below the line show the formulae that will be added to the branch. The vertical line in the $[\neg \land]$ rule means that the branch splits to the left and right.
  \caption{The rules of the tableau calculus for HPL.}
  \label{figrules}
  \end{figure}
  \label{defrules}
\end{dfn}

Most of these rules are based on the method of \cite{bolander2007}. However, since we have two dimensions, the rules related to the modal and hybrid languages are split into two rules. For example, instead of $[\dia]$ in the tableau calculus for one-dimensional hybrid logic, we have $[\dia_1]$ and $[\dia_2]$ in the tableau calculus of HPL. Note that the formula $@_i \dia_1 j$ generated by $[\dia_1]$ and the formula $@_a \dia_2 b$ generated by $[\dia_2]$ relate to two accessibility relations. Since the relations of the two dimensions are independent of each other, this formula consists only of the materials related to $\Nom_1$.

The rules $[Red_1]$ and $[Red_2]$ reflect the axioms Red$@_1$ and Red$@_2$ in \cite{sano2010}, respectively. These are:
\begin{align*}
  \text{Red$@_1$:} &\quad @_i a \leftrightarrow a, \\
  \text{Red$@_2$:} &\quad @_a i \leftrightarrow i,
\end{align*}
where $i \in \Nom_1$ and $a \in \Nom_2$. Using these rules, hybrid formulae not only of the form $@_i @_a \varphi$ but also of the form $@_i \varphi$ or $@_a \varphi$ occur in a tableau. Then, we need four $[Id]$ rules for each dimension and each formula, single or double ``prefix.''

\begin{rem}
  Considering that $\square_1 \varphi \equiv \neg \diamondsuit_1 \neg \varphi$, the rule $[\neg \diamondsuit_1]$ is equivalent to the following rule:
  \[
    \infer[{[\square_1]}]
      {@_j @_a \varphi}
      {\deduce{@_i \diamondsuit_1 j}{@_i @_a \square_1 \varphi}} .
  \] 
  We can say the same for $\square_2$. Moreover, we can create the rules for Boolean operators $\lor$ and $\rightarrow$, using the facts that $\varphi \lor \psi \equiv \neg (\neg \varphi \land \neg \psi)$ and $\varphi \rightarrow \psi \equiv \neg (\varphi \land \neg \psi)$. If we can also use these rules, then we no longer have to replace these symbols in a formula with primitive ones. 
  \label{remsquare}
\end{rem}

\begin{dfn}[provability]
  Given a formula $\varphi$, we say that $\varphi$ is \emph{provable} and write $\vdash \varphi$ if there is a closed tableau whose root formula is $@_i @_a \neg \varphi$, where both $i \in \Nom_1$ and $a \in \Nom_2$ do not occur in $\varphi$.
  
  Also, given two finite sets of formulae $\Gamma$ and $\Delta$, we write $\Gamma \vdash \Delta$ if there is a closed tableau whose root formula is $@_i @_a \neg (\bigwedge \Gamma \rightarrow \bigvee \Delta)$, equivalent to $@_i @_a (\bigwedge \Gamma \land \neg(\bigvee \Delta))$, where both $i \in \Nom_1$ and $a \in \Nom_2$ do not occur in any formula in $\Gamma$ or $\Delta$.
\end{dfn}

\section{Soundness}
\label{secsound}

In this section, we show that the tableau calculus of HPL is sound for the class of product frames, i.e., if $\vdash \varphi$, then $\models \varphi$.

First, we need to make some preparations. Let us denote by $\Theta^n$ a fragment of a branch $\Theta$ that we can get by using rules $n$ times from the root formula. 

Next, we define the concept \emph{faithfulness} for a branch $\Theta$ and a model $\mathfrak{M}$. This is named after \cite[p. 31, Definition 2.9.2]{priest2008}.

\begin{dfn}[faithful model]
Given a branch (or its fragment) $\Theta$ and a model $\mathfrak{M} = (\mathfrak{F}, V)$, we say that $\mathfrak{M}$ is \emph{faithful} to $\Theta$ if all of the following three conditions hold:
  \begin{enumerate}[i)]
    \item If $@_j @_b \psi \in \Theta$, then $\mathfrak{M}, (j^V, b^V) \models \psi$. \label{sound1}
    \item If $@_j s \in \Theta \ (s = k, \neg k, \diamondsuit_1 k)$, then for all $y \in W_2$, $\mathfrak{M}, (j^V, y) \models s$.\label{sound2}
    \item If $@_b t \in \Theta \ (t = c, \neg c, \diamondsuit_2 c)$, then for all $x \in W_1$, $\mathfrak{M}, (x, b^V) \models t$.\label{sound3}
  \end{enumerate}  
  \label{deffaithful}
\end{dfn}

First, we prove this lemma. This is a ``step case'' of Lemma \ref{lemsound2}.

\begin{lem}
  Let $\Theta = \Theta^m$ be a fragment of a branch and $\mathfrak{M} = (\mathfrak{F}, V)$ a model that is faithful to $\Theta^m$. Let $\Theta^{m+1}$ be a new branch (possibly not the only one) that we acquire by applying one of the rules to $\Theta^m$. Then, there exists a model $\mathfrak{M}'=(\mathfrak{F},V')$ such that $\mathfrak{M}'$ is faithful to $\Theta^{m+1}$.
  \label{lemsound}
\end{lem}
Then, there exists a model $\mathfrak{M}'=(\mathfrak{F},V')$ such that $\mathfrak{M}'$ is faithful to $\Theta^{m+1}$.
\begin{proof}
  This proof is divided into different cases depending on which rule is applied to $\Theta^m$. We give the proof only for the cases $[\diamondsuit_1]$ and $[@_1]$. The other cases are left to the reader.

  \begin{description}
    \item[{$[\diamondsuit_1]$}] Suppose that we acquire two new formulae $@_j \diamondsuit_1 k$ and $@_k @_b \psi$ by applying $[\diamondsuit_1]$ to $\Theta^m$. Let $\Theta^{m+1}$ be a branch that contains both formulae. Since we can apply $[\diamondsuit_1]$ to $\Theta^m$, $@_j @_b \diamondsuit_1 \psi \in \Theta^m$ holds, and by assumption $\mathfrak{M}, (j^V, b^V) \models \diamondsuit_1 \psi$ holds. This implies that there is some $x' \in W_1$ such that $j^V R_1 x'$ and $\mathfrak{M}, (x', b^V) \models \psi$. Now, we define a new valuation function $V'$:
    \[
     V'(s) = 
     \begin{cases}
       \{ x' \} \times W_2 & (s = k) \\
       V(s) & (\text{otherwise}) .
     \end{cases} 
    \]
    Define $\mathfrak{M}' = (\mathfrak{F}, V')$. Then, we have $j^{V'} R_1 k^{V'}$ and $\mathfrak{M}', (k^{V'}, b^{V'}) \models \psi$. By the latter, Definition \ref{deffaithful} (\ref{sound1}) holds for $@_k @_b \psi \in \Theta^{m+1}$ and $(k^{V'}, b^{V'})$. By the former, $\mathfrak{M}', (j^{V'}, y) \models \diamondsuit_1 k$ holds for an arbitrary $y$. Then, Definition \ref{deffaithful} (\ref{sound2}) holds for $@_j \diamondsuit_1 k \in \Theta^{m+1}$ and $j^{V'}$. Moreover, since $k$ does not occur in $\Theta^m$, the interpretation of any formula in $\Theta^m$ does not change even if $V$ is changed into $V'$. Therefore, $\mathfrak{M}'$ is faithful to $\Theta^{m+1}$.
    
    \item[{$[@_1]$}] Suppose that we acquire a formula $@_k @_b \psi$ by applying $[@_1]$ to $\Theta^m$. Let $\Theta^{m+1}$ be a branch that contains the formula. Since we can apply $[@_1]$ to $\Theta^m$, $@_j @_b @_k \psi \in \Theta^m$ holds. Then, by assumption, $\Kmodel, (j^V, b^V) \models @_k \psi$ holds. By this fact, we straightforwardly acquire $\Kmodel, (k^V, b^V) \models \psi$, which implies that $\mathfrak{M}$ is faithful to $\Theta^{m+1}$.
  \end{description}
\end{proof}

\begin{lem}
  Let $\mathcal{T}$ be a tableau whose root formula is $@_i @_a \varphi$ and all of whose branches are finite, and let $\mathfrak{M} = (\mathfrak{F}, V)$ be a product model such that a possible world $(i^V, a^V) \in W$ satisfies $\mathfrak{M}, (i^V, a^V) \models \varphi$. Then, there exists a pair of a branch $\Theta$ in $\mathcal{T}$ and a model $\mathfrak{M}' = (\mathfrak{F}, V')$ such that $\mathfrak{M}'$ is faithful to $\Theta$.
  \label{lemsound2}
\end{lem}
\begin{proof}
  By assumption, $\mathfrak{M}$ is faithful to the fragment of the branch $\Theta^0$ that we can get by applying the rules $0$ times to the root formula (the fragment that contains only the root formula). Using this fact and Lemma \ref{lemsound}, we can construct a model $\mathfrak{M}' = (\mathfrak{F}, V')$ that is faithful to at least one branch $\Theta$. This is the desired pair of a branch and a model.
\end{proof}

Lemma \ref{lemsound2} holds only for a finite tableau. To prove soundness, we have to show that every closed tableau is finite.

\begin{lem}
  If a tableau is closed, all the branches of it are finite.
  \label{lemsound3}
\end{lem}
\begin{proof}
  Take an arbitrary branch $\Theta$ of a closed tableau. Then, there are formulas $@_i @_a \varphi$ and $@_i @_a \varphi$ in $\Theta$ ($@_i @_a \varphi$ may have the form $@_i \varphi$ or $@_a \varphi$, but this does not affect the proof.) By the construction of tableau (Definition \ref{deftableau}), one of the formulas should be at the end of the branch. Therefore, $\Theta$ is finite.
\end{proof}

\begin{thm}[soundness]
  The tableau calculus of HPL is sound for the class of product frames, i.e., if $\vdash \varphi$, then $\models \varphi$.
  \label{thmsound}
\end{thm}
\begin{proof}
  By \textit{reductio ad absurdum}. Assume that both $\vdash \varphi$ and $\not \models \varphi$. From the latter, there are a model $\mathfrak{M} = (\mathfrak{F}, V)$ and its possible world $(x, y)$ such that $\mathfrak{M}, (x, y) \models \neg \varphi$. Now, take arbitrary nominals $i \in \Nom_1$ and $a \in \Nom_2$ both of which do not occur in $\varphi$. Also, define $V'$ as follows:
  \[
    V'(s) = 
    \begin{cases}
      \{ x \} \times W_2 & (s = i) \\
      W_1 \times \{ y \} & (s = a) \\
      V(s) & (\text{otherwise}) .
    \end{cases}
  \]
  Thus, for $\mathfrak{M}' = (\mathfrak{F}, V')$, we have $\mathfrak{M}', (i^{V'}, a^{V'}) \models \neg \varphi$. Now, the assumption that $\vdash \varphi$ implies that there is a tableau $\mathcal{T}$ whose root formula is $@_i @_a \neg \varphi$. By Lemma \ref{lemsound3}, all branches in $\mathcal{T}$ are finite. Then, by Lemma \ref{lemsound2}, there is a pair of a branch $\Theta$ of the tableau and a model $\mathfrak{M}'' = (\mathfrak{F}, V'')$ such that $\mathfrak{M}''$ is faithful to $\Theta$.

  Since $\mathcal{T}$ is closed, so is $\Theta$. Then, at least one of the following statements holds:
  \begin{enumerate}
    \item There are a pair $(j, b)$ of nominals and a formula $\psi$ such that $@_j @_b \psi, @_j @_b \neg \psi \in \Theta$. \label{cl1}
    \item There are a nominal $j \in \Nom_1$ and a formula $\psi$ such that $@_j \psi, @_j \neg \psi \in \Theta$. \label{cl2}
    \item There are a nominal $b \in \Nom_2$ and a formula $\psi$ such that $@_b \psi, @_b \neg \psi \in \Theta$. \label{cl3}
  \end{enumerate}
  Suppose that case \ref{cl1} holds. Then, by Definition \ref{deffaithful} (\ref{sound1}), both $\mathfrak{M}'', (j^{V''}, b^{V''}) \models \psi$ and $\mathfrak{M}'', (j^{V''}, b^{V''}) \models \neg \psi$ hold, but that is a contradiction. We can similarly show a contradiction in cases \ref{cl2} and \ref{cl3}.
\end{proof}

\begin{cor}
  Given two finite sets of formulae $\Gamma$ and $\Delta$, if $\Gamma \vdash \Delta$, then $\Gamma \models \Delta$. 
\end{cor}
\begin{proof}
  By Theorem \ref{thmsound} and letting $\varphi = \bigwedge \Gamma \rightarrow \bigvee \Delta$.
\end{proof}

\section{Completeness}
\label{seccomplete}

In this section, we show that the tableau calculus of HPL is complete for the class of product frames, i.e., if $\models \varphi$, then $\vdash \varphi$.

First, we define an important concept, \emph{quasi-subformula}.

\begin{dfn}[quasi-subformula]
  Given two formulae $@_i @_a \varphi$ and $@_j @_b \psi$, we say that $@_i @_a \varphi$ is \emph{a quasi-subformula} of $@_j @_b \psi$ if one of the following statements holds:
  \begin{itemize}
    \item $\varphi$ is a subformula of $\psi$.
    \item $\varphi$ has the form $\neg \chi$, and $\chi$ is a subformula of $\psi$.
  \end{itemize}
\end{dfn}

\begin{lem}[quasi-subformula property]
  Let $\mathcal{T}$ be a tableau. If a formula in $\mathcal{T}$ has the form $@_i @_a \varphi$, then it is a quasi-subformula of the root formula of $\mathcal{T}$.
  \label{lemquasi}
\end{lem}
\begin{proof}
  By induction on the rules of tableau calculus, except these six rules that do not generate a formula of the form $@_i @_a \varphi$: $[Red_1], [Red_2], [\neg_1], [\neg_2], [Id'_1],$ and $[Id'_2]$. We provide the proof only for the cases $[\neg \land]$ and $[\diamondsuit_1]$. The other cases are left to the reader.
  \begin{description}
    \item[{$[\neg \land]$}] $@_i @_a \neg \varphi$ and $@_i @_a \neg \psi$ are both quasi-subformulae of $@_i @_a \neg (\varphi \land \psi)$. Also, $@_i @_a \neg (\varphi \land \psi)$ is a quasi-subformula of the root formula by the induction hypothesis. Therefore, $@_i @_a \neg \varphi$ and $@_i @_a \neg \psi$ are quasi-subformulae of the root formula.
    \item[{$[\diamondsuit_1]$}] It suffices to show that $@_j @_a \varphi$ is a quasi-subformula of the root formula. $@_j @_a \varphi$ is a quasi-subformula of $@_i @_a \diamondsuit_1 \varphi$, and by the induction hypothesis, $@_i @_a \diamondsuit_1 \varphi$ is a quasi-subformula of the root formula.
  \end{description}
\end{proof}

Next, we define the concept of saturation. Intuitively, a branch is saturated if we cannot add any new formula to it even if we can apply some rules to it. The following is the formal definition of that.

\begin{dfn}
  A branch $\Theta$ of a tableau is \emph{saturated} if $\Theta$ satisfies all the conditions below:

  \begin{enumerate}[i)]
    \item If $@_i @_a \neg \neg \varphi \in \Theta$, then $@_i @_a \varphi \in \Theta$. \label{satnegneg}
    \item If $@_i @_a (\varphi \land \psi) \in \Theta$, then $@_i @_a \varphi, @_i @_a \psi \in \Theta$. \label{satand}
    \item If $@_i @_a \neg (\varphi \land \psi) \in \Theta$, then either $@_i @_a \neg \varphi \in \Theta$ or $@_i @_a \neg \psi \in \Theta$. \label{satnegand}
    \item If $@_i @_a \diamondsuit_1 \varphi \in \Theta$, then there is some $j \in \textbf{Nom}_1$ such that $@_i \diamondsuit_1 j \in \Theta, @_j @_a \varphi \in \Theta$. \label{satdia1}
    \item If $@_i @_a \diamondsuit_2 \varphi \in \Theta$, then there is some $b \in \textbf{Nom}_2$ such that $@_a \diamondsuit_2 b, @_i @_b \varphi \in \Theta$. \label{satdia2}
    \item If $@_i @_a \neg \diamondsuit_1 \varphi \in \Theta$ and $@_i \diamondsuit_1 j \in \Theta$, then $@_j @_a \neg \varphi \in \Theta$. \label{satnegdia1}
    \item If $@_i @_a \neg \diamondsuit_2 \varphi \in \Theta$ and $@_a \diamondsuit_2 b \in \Theta$, then $@_i @_b \neg \varphi \in \Theta$. \label{satnegdia2}
    \item If $@_i @_a @_j \varphi \in \Theta$, then $@_j @_a \varphi \in \Theta$. \label{satat1}
    \item If $@_i @_a @_b \varphi \in \Theta$, then $@_i @_b \varphi \in \Theta$. \label{satat2}
    \item If $@_i @_a \neg @_j \varphi \in \Theta$, then $@_j @_a \neg \varphi \in \Theta$. \label{satnegat1}
    \item If $@_i @_a \neg @_b \varphi \in \Theta$, then $@_i @_b \neg \varphi \in \Theta$. \label{satnegat2}
    \item If $s = k, \neg k \ (k \in \Nom_1)$ and $@_i @_a s \in \Theta$, then $@_i s \in \Theta$. \label{satred1}
    \item If $t = c, \neg c \ (c \in \Nom_2)$ and $@_i @_a t \in \Theta$, then $@_a t \in \Theta$. \label{satred2}
    \item If $@_i \neg j \in \Theta$, then $@_j j \in \Theta$. \label{satneg1}
    \item If $@_a \neg b \in \Theta$, then $@_b b \in \Theta$. \label{satneg2}
    \item If $@_i @_a \varphi \in \Theta$ and $@_i j \in \Theta$, then $@_j @_a \varphi \in \Theta$. \label{satid11}
    \item If $@_i @_a \varphi \in \Theta$ and $@_a b \in \Theta$, then $@_i @_b \varphi \in \Theta$. \label{satid12}
    \item If $s = k, \neg k \ (k \in \Nom_1)$, and $@_i s, @_i j \in \Theta$, then $@_j s \in \Theta$. \label{satid21}
    \item If $t = c, \neg c \ (c \in \Nom_2)$, and $@_a t, @_a b \in \Theta$, then $@_b t \in \Theta$. \label{satid22}
  \end{enumerate}
  We call a tableau \emph{saturated} if all branches of it are saturated.
  \label{defsat}
\end{dfn}

The next lemma claims that for an arbitrary root formula, we can make any open branch saturated.

\begin{lem}
  Given an HPL formula $\varphi$, we can construct a tableau whose root formula is $@_i @_a \varphi$ ($i$ and $a$ do not occur in $\varphi$) and of which all the open branches are saturated.
  \label{lemconst}
\end{lem}
\begin{proof}
  Given an HPL formula $\varphi$, we construct a tableau $\mathcal{T}_n,$ as follows:
  \begin{itemize}
    \item $\mathcal{T}_0$ is a tableau with only the node $@_i @_a \neg \varphi$ ($i$ and $a$ do not occur in $\varphi$).
	\item $\mathcal{T}_{n+1}$ is constructed by applying rules to all nodes in each branch of $\mathcal{T}_n$ whenever it is possible. In constructing it, we have to observe the following rules:
	\begin{enumerate}[i)]
	  \item If the new formulae generated by applying some rules already exist in the same branch, they are not added to the branch.
	  \item If a branch becomes closed, we can no longer add any formula.
    \end{enumerate}
  \end{itemize}
  Now, we define $\mathcal{T}^*$ as the closure of all $\mathcal{T}_n \ (n \in \mathbb{N})$. Then, all the open branches of $\mathcal{T}^*$ are saturated; to prove that, suppose otherwise. Then, there exists a branch $\Theta$ such that one of the conditions of Definition \ref{defsat} does not hold. Let $@_j @_b \varphi$ (it may be of the form $@_j \psi$ or $@_b \psi$) be the formula that falsifies the condition of Definition \ref{defsat} with respect to $\Theta$. Now, we can find an integer $k \in \mathbb{N}$ such that $@_j @_b \varphi$ is in $\mathcal{T}_k$. Then, $\mathcal{T}_{k+1}$ is a tableau of which all the open branches containing $@_j @_b \varphi$ satisfy all the conditions of Definition \ref{defsat} with respect to it. Thus, so is $\mathcal{T}^*$, but that is a contradiction.
\end{proof}

Next, we define the \emph{right nominal}. Given a branch $\Theta$ of tableau, $s \in \textbf{Nom}$ is a right nominal of $\Theta$ if there is a nominal $t$ such that $@_t s \in \Theta$.

Right nominals have the following property.

\begin{lem}
  All the right nominals occurring in branch $\Theta$ occur in the root formula.
\end{lem}
\begin{proof}
  We prove the following proposition by induction on $n$, which straightforwardly derives the desired lemma:
  \begin{quotation}
    For any nominal $s \in \Nom$, $n \in \mathbb{N}$, and a fragment of branch $\Theta^n$, if $@_t s \in \Theta^n$ or $@_t \neg s \in \Theta^n$, then $s$ occurs in $\Theta$.
  \end{quotation}
  Since $@_t s$ and $@_t \neg s$ can be derived using $[Id'_1]$, $[Red_1]$, or $[\neg_1]$, then we have to check only these three cases. We assume that $s, t \in \Nom_1$ (the case in which $s, t \in \Nom_2$ is the same).
  \begin{enumerate}[i)]
    \item Suppose that $@_t s$ is derived using $[Id'_1]$ (we can do it in the same way for $@_t \neg s$). Then, there is a nominal $u \in \Nom_1$ such that $@_u s \in \Theta^{n-1}$ and $@_u t \in \Theta^{n-1}$. \label{right1}
    \item Suppose that $@_t s$ is derived using $[Red_1]$ (we can do it in the same way for $@_t \neg s$). Then, there is a nominal $a \in \Nom_2$ such that $@_t @_a s \in \Theta^{n-1}$. \label{right2}
    \item If $@_t s$ is derived using $[\neg_1]$, then $t = s$, and there is a nominal $u \in \Nom_1$ such that $@_u \neg s \in \Theta^{n-1}$. \label{right3}
  \end{enumerate}
\end{proof}

\begin{dfn}
  Let $\Theta$ be a branch of a tableau. For any right nominals $i, j \in \Nom_1$ and $a, b \in \Nom_2$ occurring in $\Theta$, we define the binary relations $\sim_{\Theta}^1$ and $\sim_{\Theta}^2$ as follows:
  \begin{align*}
    i \sim_{\Theta}^1 j \quad &\text{iff} \quad @_i j \in \Theta,  \\
    a \sim_{\Theta}^2 b \quad &\text{iff} \quad @_a b \in \Theta.
  \end{align*}
  \label{defsim}
\end{dfn}

\begin{lem}
 Let $\Theta$ be a saturated branch of a tableau. Then, $\sim_{\Theta}^k (k = 1, 2)$ of Definition \ref{defsim} are equivalence relations of $\Nom_1$ and $\Nom_2$, respectively.
  \label{lemsim}
\end{lem}
\begin{proof}
  It suffices to show that $\sim_{\Theta}^1$ is reflexive, symmetric, and transitive.
  
  First, we show reflexivity. Let $i$ be a right nominal. Then, there is a nominal $j$ such that $@_j i \in \Theta$. Since $\Theta$ is saturated, by Definition \ref{defsat}(\ref{satid21}) we have $@_i i \in \Theta$. Therefore, $i \sim_{\Theta}^1 i$.

  Second, we show symmetry. Let $i \sim_{\Theta}^1 j$. By definition, $@_i j \in \Theta$ holds. Also, by reflexivity of $\sim_{\Theta}^1$, we have $@_i i \in \Theta$. Since $\Theta$ is saturated, by Definition \ref{defsat}(\ref{satid21}) we have $@_j i \in \Theta$. Therefore, $j \sim_{\Theta}^1 i$.

  Lastly, we show transitivity. Let $i \sim_{\Theta}^1 j$ and $j \sim_{\Theta}^1 k$. By definition, $@_i j \in \Theta$ and $@_j k \in \Theta$ hold. We have $@_j i \in \Theta$ by the symmetry of $\sim_{\Theta}^1$, and since $\Theta$ is saturated, by Definition \ref{defsat}(\ref{satid21}) we have $@_i k \in \Theta$. Therefore, we have $i \sim_{\Theta}^1 k$.
\end{proof}

\begin{dfn}
  Let $\Theta$ be a saturated branch of a tableau, and let $i \in \Nom_1$ and $a \in \Nom_2$ be right nominals occurring in $\Theta$. We denote by $[i]_{\Theta}^1$ an equivalence class of a right nominal $i$ by the relation $\sim_{\Theta}^1$. Similarly, we denote by $[a]_{\Theta}^2$ an equivalence class of a right nominal $a$ by the relation $\sim_{\Theta}^2$.
\end{dfn}

Note that since $\Nom_1$ and $\Nom_2$ are both countably infinite, we can insert well-order in each of them. For instance, Bolander and Blackburn (2007) \cite{bolander2007} established an order that depends on when it was introduced\footnote{They did not mention the case that two or more nominals are introduced at the same formula (it will happen when we put root formula on.) In this case, however, we can define an easy and natural order: the more left, the earlier.}. For a non-empty set $A$ of nominals, we write $\min(A)$ as the minimal element in this order. We use this notation $\min(A)$ in Definition \ref{defur}.

\begin{dfn}
  Let $\Theta$ be a saturated branch of a tableau and $s$ a nominal occurring in $\Theta$ (note that $s$ does not have to be a right nominal). We define an \emph{urfather} of $s$ on $\Theta$ (denoted by $u_{\Theta}(s)$) as follows:
  \[
    u_{\Theta}(s) =
    \begin{cases}
      \min([j]_\Theta^1) & \text{if} \quad s \in \Nom_1 \ \text{and} \ @_s j \in \Theta \\
      \min([b]_\Theta^2) & \text{if} \quad s \in \Nom_2 \ \text{and} \ @_s b \in \Theta \\
      s & \text{otherwise.}
    \end{cases}
  \]
  \label{defur}
\end{dfn}

\begin{prop}
  $u_{\Theta}(s)$ of Definition \ref{defur} is well-defined.
\end{prop}
\begin{proof}
  Assume that $@_s i, @_s j \in \Theta$ and $s, i, j \in \Nom_1$ (we can do the case $\Nom_2$ in the same way). Since $\Theta$ is saturated, by Definition \ref{defsat}(\ref{satid21}) we have $@_i j \in \Theta$. Thus, $[i]_{\Theta}^1 = [j]_{\Theta}^1$, and we have $\min([i]_{\Theta}^1) = \min([j]_{\Theta}^1)$. Therefore, $u_{\Theta}(s)$ is unique.
\end{proof}

\begin{lem}
  Let $\Theta$ be a saturated branch of a tableau. Then, we have the following properties: 
  \begin{enumerate}[i)]
    \item If $s \in \Nom$ is a right nominal occuring in $\Theta$, then $@_{u_\Theta(s)} s \in \Theta$. \label{lemurright}
    \item For all $s, t \in \Nom$, if $@_s t \in \Theta$, then $u_\Theta(s) = u_\Theta(t)$. \label{lemureq}
    \item For all quasi-subformulae $@_i @_a \varphi \in \Theta$ of the root formula in $\Theta$, we have $@_{u_\Theta(i)} @_{u_\Theta(a)} \varphi \in \Theta$. \label{lemurcl}
  \end{enumerate}
  \label{lemurpro}  
\end{lem}

\begin{proof}
  \begin{enumerate}[i)]
    \item If $s$ is a right nominal, then by Lemma \ref{lemsim} $@_s s \in \Theta$ holds. Thus, by definition we have $u_{\Theta}(s) = \min([s]_{\Theta}^k) \ (k = 1, 2)$. From this, we have $u_\Theta(s) \sim_\Theta^k s$. Therefore, $@_{u_\Theta(s)} s \in \Theta$.
    \item First, by $@_s t \in \Theta$, we have $u_{\Theta}(s) = \min([t]_{\Theta}^k) \ (k = 1, 2)$. Also, since $t$ is a right nominal, $u_{\Theta}(t) = \min([t]_{\Theta}^k)$ holds by doing the same as in (\ref{lemurright}). Therefore, $u_{\Theta}(s) = u_{\Theta}(t)$.
    \item Suppose that $@_i @_a \varphi \in \Theta$. Then, we have the following four cases:
    \begin{enumerate}[a)]
      \item $u_{\Theta}(i) = i, u_{\Theta}(a) = a$. \label{urcl1}
      \item $u_{\Theta}(i) = i, u_{\Theta}(a) = \min([b]_\Theta^2)$. \label{urcl2}
      \item $u_{\Theta}(i) = \min([j]_\Theta^1), u_{\Theta}(a) = a$. \label{urcl3}
      \item $u_{\Theta}(i) = \min([j]_\Theta^1), u_{\Theta}(a) = \min([b]_\Theta^2)$. \label{urcl4}
    \end{enumerate}
    Case (\ref{urcl1}) is trivial. In case (\ref{urcl4}), by definition we have $@_i j, @_a b\in \Theta$. Since $\Theta$ is saturated, by Definition \ref{defsat} (\ref{satid11}) and (\ref{satid12}), we have $@_j @_b \varphi \in \Theta$. Also, because $u_{\Theta}(i) = \min([j]_\Theta^1)$ and $u_{\Theta}(a) = \min([b]_\Theta^2)$, $j \sim_\Theta^1 u_\Theta(i)$ and $b \sim_{\Theta}^2 u_{\Theta}(a)$. Hence, $@_j u_\Theta(i), @_b u_{\Theta}(a) \in \Theta$ holds. By those facts, using the saturation of $\Theta$ and Definition \ref{defsat}(\ref{satid11}), (\ref{satid12}), we have $@_{u_{\Theta}(i)} @_{u_{\Theta}(a)} \varphi \in \Theta$. We can do the same in cases (\ref{urcl2}) and (\ref{urcl3}).
  \end{enumerate}
\end{proof}

To prove completeness, we construct a model from a tableau. In this paper, we take only urfathers of nominals occurring in a branch as possible worlds, which differs from the method in \cite{bolander2007}. Note that taking all nominals as possible worlds also works.

\begin{dfn}
  Given an open branch $\Theta$ of a tableau, the product model $\mathfrak{M}^{\Theta} = (W^{\Theta}, R_h^{\Theta}, R_v^\Theta, V^{\Theta})$ is defined as follows:
  \begin{align*}
    W_1^{\Theta} &= \{ u_\Theta(i) \in \Nom_1 \mid i \ \text{occurs in} \ \Theta \} , \\
    W_2^{\Theta} &= \{ u_\Theta(a) \in \Nom_2 \mid a \ \text{occurs in} \ \Theta \} , \\
    W^{\Theta} &= W_1^\Theta \times W_2^\Theta , \\
    R_h^{\Theta} &= \{ ((u_{\Theta}(i), u_{\Theta}(a)), (u_{\Theta}(j), u_{\Theta}(a))) \mid @_i \diamondsuit_1 j \in \Theta \} , \\
    R_v^{\Theta} &= \{ ((u_{\Theta}(i), u_{\Theta}(a)), (u_{\Theta}(i), u_{\Theta}(b))) \mid @_a \diamondsuit_2 b \in \Theta \} , \\
    V^{\Theta}(p) &= \{ (u_\Theta(i), u_\Theta(a)) \mid @_i @_a p \in \Theta \} , \quad \text{where} \quad p \in \Prop , \\
    V^{\Theta}(i) &= \{ (u_{\Theta}(i) \} \times W_2 , \quad \text{where} \quad i \in \Nom_1 , \\
    V^{\Theta}(a) &= W_1 \times \{ u_{\Theta}(a) \} ,  \quad \text{where} \quad a \in \Nom_2 .
  \end{align*}
\end{dfn}

\begin{lem}[model existence theorem]
  Let $\Theta$ be an open saturated branch of a tableau. Then, we have
    \[
      \text{if } @_i @_a \varphi \in \Theta, \text{ then } \mathfrak{M}^{\Theta}, (u_\Theta(i), u_{\Theta}(a)) \models \varphi.
    \]
  \label{lemmodelexist}
\end{lem}
\begin{proof}
  By induction on the complexity of $\varphi$.
  
  \begin{description}
    \item[{$[\varphi = p]$}] Suppose that $@_i @_a p \in \Theta$. Then, by the definition of $V^{\Theta}$, $ (u_{\Theta}(i), u_{\Theta}(a)) \in V^{\Theta}(p)$, from which it follows that $\mathfrak{M}^{\Theta}, (u_\Theta(i), u_{\Theta}(a)) \models p$.

    \item[{$[\varphi = \neg p]$}] Suppose that $@_i @_a \neg p \in \Theta$. Since $\Theta$ is open, $@_i @_a p \notin \Theta$. Then, by the definition of $V^{\Theta}$, $(u_\Theta(i), u_{\Theta}(a)) \notin V^{\Theta}(p)$ holds. Therefore, $\mathfrak{M}^{\Theta}, (u_\Theta(i), u_{\Theta}(a)) \models \neg p$.

    \item[{$[\varphi = j]$}] Suppose that $@_i @_a j \in \Theta \ (j \in \Nom_1)$. Since $\Theta$ is saturated, by Definition \ref{defsat} (\ref{satred1}) we have $@_i j \in \Theta$. Then, by Lemma \ref{lemurpro} (\ref{lemureq}), $u_{\Theta}(i) = u_{\Theta}(j)$ holds. Thus, $(u_{\Theta}(i), u_\Theta(a)) \in V^{\Theta}(j)$. Therefore, $\mathfrak{M}^{\Theta}, (u_\Theta(i), u_{\Theta}(a)) \models j$. We can do the same in the case $\varphi = b \in \Nom_2$.

    \item[{$[\varphi = \neg j]$}] Suppose that $@_i @_a \neg j \in \Theta \ (j \in \Nom_1)$. By Lemma \ref{lemurpro} (\ref{lemurcl}) we have $@_{u_\Theta(i)} @_{u_\Theta(a)} \neg j \in \Theta$. Since $\Theta$ is saturated, by Definition \ref{defsat} (\ref{satred1}) and (\ref{satneg1}), $@_{u_{\Theta}(i)} \neg j, @_j j \in \Theta$. From the latter, $j$ is a right nominal. Thus, by Lemma \ref{lemurpro} (\ref{lemurright}) we have $@_{u_{\Theta}(j)} j \in \Theta$. Hence, $@_{u_{\Theta}(i)} \neg j, @_{u_{\Theta}(j)} j \in \Theta$. Since $\Theta$ is open, $u_{\Theta}(i) \neq u_{\Theta}(j)$. Then, by the definition of $V^{\Theta}$ it follows that $(u_\Theta(i), u_{\Theta}(a)) \notin V^{\Theta}(j)$. Therefore, $\mathfrak{M}^{\Theta}, (u_\Theta(i), u_{\Theta}(a)) \models \neg j$. We can do the same in the case $\varphi = \neg b \ (b \in \Nom_2)$.

    \item[{$[\varphi = (\psi_1 \land \psi_2)]$}] Suppose that $@_i @_a (\psi_1 \land \psi_2) \in \Theta$. Since $\Theta$ is saturated, by Definition \ref{defsat} (\ref{satand}) we have $@_i @_a \psi_1, @_i @_a \psi_2 \in \Theta$. Thus, by the induction hypothesis, both $\mathfrak{M}^{\Theta}, (u_\Theta(i), u_{\Theta}(a)) \models \psi_1$ and $\mathfrak{M}^{\Theta}, (u_\Theta(i), u_{\Theta}(a)) \models \psi_2$ hold. Therefore, by the definition of $\models$, we have $\mathfrak{M}, (u_\Theta(i), u_{\Theta}(a)) \models \psi_1 \land \psi_2$.

    \item[{$[\varphi = \neg (\psi_1 \land \psi_2)]$}] Suppose that $@_i @_a \neg (\psi_1 \land \psi_2) \in \Theta$. Since $\Theta$ is saturated, by Definition \ref{defsat} (\ref{satnegand}) we have either $@_i @_a \neg \psi_1 \in \Theta$ or $@_i @_a \neg \psi_2 \in \Theta$. Then, by the induction hypothesis, either  $\mathfrak{M}^{\Theta}, (u_\Theta(i), u_{\Theta}(a)) \models \neg \psi_1$ or $\mathfrak{M}^{\Theta}, (u_\Theta(i), u_{\Theta}(a)) \models \neg \psi_2$ holds. Therefore, $\mathfrak{M}, (u_\Theta(i), u_{\Theta}(a)) \models \neg (\psi_1 \land \psi_2)$.

    \item[{$[\varphi = \diamondsuit_1 \psi]$}] Suppose that $@_i @_a \diamondsuit_1 \psi \in \Theta$. Since $\Theta$ is saturated, by Definition \ref{defsat} (\ref{satdia1}) there is a nominal $j \in \Nom_1$ such that $@_i \diamondsuit_1 j, @_j @_a \psi \in \Theta$. From the former and the definition of $R_h^{\Theta}$, we have $(u_{\Theta}(i), u_{\Theta}(a)) R_h^\Theta (u_{\Theta}(j), u_\Theta(a))$, and by applying the induction hypothesis to the latter, we have $\mathfrak{M}^{\Theta}, (u_\Theta(j), u_{\Theta}(a)) \models \psi$. Therefore, $\mathfrak{M}^{\Theta}, (u_\Theta(i), u_{\Theta}(a)) \models \diamondsuit_1 \psi$. We can do the same in the case $\varphi = \diamondsuit_2 \psi$.

    \item[{$[\varphi = \neg \diamondsuit_1 \psi]$}] Suppose that $@_i @_a \neg \diamondsuit_1 \psi \in \Theta$. Now, assume that there is a nominal $j \in \Nom_1$ such that $(u_\Theta(i), u_\Theta(a)) R_h^\Theta (u_\Theta(j), u_\Theta(a))$ and $@_i \dia_1 j \in \Theta$ (if such a $j$ does not exist, in the model $\mathfrak{M}^\Theta$ there is no possible world that is reachable from $(u_\Theta(i), u_\Theta(a))$ through $R_h^\Theta$). Thus, $\mathfrak{M}^{\Theta}, (u_\Theta(i), u_{\Theta}(a)) \models \neg \diamondsuit_1 \psi$.) If we take such a $j$, since $\Theta$ is saturated, by Definition \ref{defsat} (\ref{satnegdia1}) we have $@_j @_a \neg \psi \in \Theta$. Hence, by the induction hypothesis, $\mathfrak{M}^{\Theta}, (u_\Theta(j), u_\Theta(a)) \models \neg \psi$. This holds for all $j$, so we have $\mathfrak{M}^{\Theta}, (u_\Theta(i), u_\Theta(a)) \models \square_1 \neg \psi$. Therefore, $\mathfrak{M}^{\Theta}, (u_\Theta(i), u_\Theta(a)) \models \neg \diamondsuit_1 \psi$. We can do the same in the case $\varphi = \neg \diamondsuit_2 \psi$.

    \item[{$[\varphi = @_j \psi]$}] Suppose that $@_i @_a @_j \psi \in \Theta \ (j \in \Nom_1)$. Since $\Theta$ is saturated, by Definition \ref{defsat} (\ref{satat1}) we have $@_j @_a \psi \in \Theta$. By the induction hypothesis, $\mathfrak{M}^\Theta, (u_\Theta(j), u_\Theta(a)) \models \psi$ holds. Therefore, $\mathfrak{M}^\Theta, (u_\Theta(i), u_\Theta(a)) \models @_j \psi$. We can do the same in the case $\varphi = @_b \psi \ (b \in \Nom_2)$.

    \item[{$[\varphi = \neg @_j \psi]$}] Suppose that $@_i @_a \neg @_j \psi \in \Theta \ (j \in \Nom_1)$. Since $\Theta$ is saturated, by Definition \ref{defsat} (\ref{satnegat1}) we have $@_j @_a \neg \psi \in \Theta$. By the induction hypothesis $\mathfrak{M}^\Theta, (u_\Theta(j), u_\Theta(a)) \models \neg \psi$ holds. Then, $\mathfrak{M}^\Theta, (u_\Theta(i), u_\Theta(a)) \models @_j \neg \psi$. Therefore, $\mathfrak{M}^\Theta, (u_\Theta(i), u_\Theta(a)) \models \neg @_j \psi$. We can do the same in the case $\varphi = \neg @_b \psi \ (b \in \Nom_2)$.
  \end{description}
\end{proof}

\begin{thm}[completeness]
  The tableau calculus of HPL is complete for the class of product frames, i.e., if $\models \varphi$, then $\vdash \varphi$.
  \label{thmcomplete}
\end{thm}
\begin{proof}
  We show the contraposition.

  Suppose that $\not \vdash \varphi$. If we choose two arbitrary nominals $i \in \Nom_1$ and $a \in \Nom_2$ such that $i$ and $a$ do not occur in $\varphi$, then all tableaux whose root formula is $@_i @_a \neg \varphi$ are not closed. Thus, a saturated tableau $\mathcal{T}$ constructed by the method of Lemma \ref{lemconst} is not closed either. This implies that we have an open and saturated branch $\Theta$ from $\mathcal{T}$. Then, by Lemma \ref{lemmodelexist}, we have $\mathfrak{M}^{\Theta}, (u_{\Theta}(i), u_{\Theta}(a)) \models \neg \varphi$. By the definition of $\models$, we have $\mathfrak{M}^{\Theta}, (u_{\Theta}(i), u_{\Theta}(a)) \not \models \varphi$, and this implies that there is a model in which $\varphi$ does not hold. Therefore, $\not \models \varphi$.
\end{proof}

\begin{cor}
  Given two finite sets of formulae $\Gamma$ and $\Delta$, if $\Gamma \models \Delta$, then $\Gamma \vdash \Delta$.
\end{cor}
\begin{proof}
  By Theorem \ref{thmcomplete} and letting $\varphi$ be $\bigwedge \Gamma \rightarrow \bigvee \Delta$.
\end{proof}

\section{Hybrid Dependent Product Logic (HdPL)}
\label{secHdPL}

A property of HPL is that two diamond operators are independent of each other. What if, for example, one diamond depends on the other? In the following, we deal with that logic, called hybrid dependent product logic (HdPL). 

\subsection{Frames and Models}

A frame in HdPL is called a \emph{dependent product Kripke frame}. This is defined as follows. 
\begin{dfn}
  Given a Kripke frame $\mathfrak{F}_1 = (W_1, R_1)$ and a multi-modal Kripke frame $\mathfrak{F}_2 = (W_2, (R_2(x))_{x \in W_1})$, a \emph{dependent product Kripke frame (d-product frame)} $\mathfrak{F} = (W, R_h, R_v)$ is defined as follows:
  \begin{itemize}
    \item $W = W_1 \times W_2$.
    \item $(x, y) R_h (x', y')$ iff $x R_1 x'$ and $y = y'$.
    \item $(x, y) R_v (x', y')$ iff $x = x'$ and $y R_2(x) y'$.
  \end{itemize}
  The definition of a \emph{dependent product Kripke model (d-product model)} is similar to Definition \ref{defmodel}.
\end{dfn}

As can be seen, the second accessibility relation $R_2$ depends on the world that we are in.

The satisfaction relation $\models$ of a d-product model is the same as Definition \ref{defsatis}, using the notations  $R_h$ and $R_v$. If we rewrite the definition without them, then the case of $\diamondsuit_2$ is different from that of HPL, as follows: 
  \[
    \mathfrak{M}, (x, y) \models \diamondsuit_2 \varphi \quad \text{iff} \quad \text{there is some} \ y' \in W_2 \ \text{s.t.} \ yR_2(x)y' \ \text{and} \ \mathfrak{M}, (x, y') \models \varphi. 
  \]

In d-product models, we lose some useful properties. For example, we cannot swap the order of diamonds freely.

\begin{prop}
  The formula $\diamondsuit_1 \diamondsuit_2 p \leftrightarrow \dia_2 \dia_1 p$ is not valid in all d-product frames.
  \label{propswap}
\end{prop}
\begin{proof}
  Consider the model $\mathfrak{M} = (W, R_h, R_v, V)$ (see Figure \ref{figdframe}):
  \begin{align*}
    W_1 &= \{ x_1, x_2 \}, \ W_2 = \{ y_1, y_2 \}, \\
    R_1 &= \{ (x_1, x_2) \}, \\
    R_2(x_1) &= \emptyset, R_2(x_2) = \{ (y_1, y_2) \}, \\
    V(p) &= \{ (x_2, y_2) \}.
  \end{align*}
  Then, $\mathfrak{M}, (x_1, y_1) \models \dia_1 \dia_2 p$ holds, because we have the following three facts: $(x_1, y_1) R_h (x_2, y_1)$, $(x_2, y_1) R_v (x_2, y_2)$, and $\mathfrak{M}, (x_2, y_2) \models p$. However, it is false that $\mathfrak{M}, (x_1, y_1) \models \dia_2 \dia_1 p$, since there is no world reachable from $(x_1, y_1)$ through $R_v$.
\end{proof}

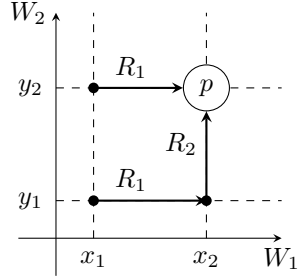
\begin{figure}[tb]
  \centering
  \begin{tikzpicture}
    \draw[->,>=stealth](3.5,0)--(7,0)node[below]{$W_1$}; 
    \draw[->,>=stealth](4,-0.5)--(4,3)node[left]{$W_2$}; 
    \draw[dashed](4.5,0)--(4.5,3);
    \draw(4.5,-0.05)node[below]{$x_1$};
    \draw[dashed](6,0)--(6,3);
    \draw(6,-0.05)node[below]{$x_2$};
    \draw[dashed](4,0.5)--(7,0.5);
    \draw(3.95,0.5)node[left]{$y_1$};
    \draw[dashed](4,2)--(7,2);
    \draw(3.95,2)node[left]{$y_2$};
    \coordinate (x1y1) at (4.5,0.5);
    \fill (x1y1) circle (2pt);
    \coordinate (x2y1) at (6,0.5);
    \fill (x2y1) circle (2pt);
    \coordinate (x1y2) at (4.5,2);
    \fill (x1y2) circle (2pt);
    \node (x2y2) [fill=white, draw, circle] at (6,2) {$p$};
    \draw[->, >=stealth, thick](x1y1)--(x2y1);
    \draw(5,0.5)node[above]{$R_1$};
    \draw[->, >=stealth, thick](x1y2)--(x2y2);
    \draw(5,2)node[above]{$R_1$};
    \draw[->, >=stealth, thick](x2y1)--(x2y2);
    \draw(6,1.25)node[left]{$R_2$};
  \end{tikzpicture}
  \caption{An example of a d-product frame in the proof of Lemma \ref{propswap}.}
  \label{figdframe}
\end{figure}

\begin{rem}
  Note that d-product frames are generalizations of product frames. In fact, by adding the following condition to any given d-product frame, it becomes a product frame: for all $x, x' \in W_1$, $R_2(x) = R_2(x')$ holds.
\end{rem}

\subsection{Tableau Calculus for HdPL}
Next, let us define the tableau calculus for HdPL. We can obtain it by modifying the tableau calculus of HPL.
First, we replace two rules, $[\dia_2]$ and $[\neg \dia_2]$, with the following two rules:

\begin{equation*}
    \infer[{[\diamondsuit_2^d]}]
      {\deduce{@_i @_b \varphi}{@_i @_a \diamondsuit_2 b}}
      {@_i @_a \diamondsuit_2 \varphi} ,
    \qquad
    \infer[{[\neg \diamondsuit_2^d]}]
      {@_i @_b \neg \varphi}
      {\deduce{@_i @_a \diamondsuit_2 b}{@_i @_a \neg \diamondsuit_2 \varphi}} .
\end{equation*}

The rule $[\dia_2^d]$ has the same restrictions as $[\dia_2]$. That is, $[\dia_2^d]$ can be applied only once per formula, and $b$ is a new nominal in $\Nom_2$. 

Moreover, we need further restrictions to obtain the HdPL tableau calculus. To describe these, we add a new concept called an \emph{accessibility formula}. 

By applying $[\dia_2^d]$, we can add two new formulae $@_i @_a \dia_2 b$ and $@_i @_b \varphi$ to the branch. We call the former formula, containing a new nominal as a right nominal, an \emph{accessibility formula}. Then, we are ready to write restrictions on the rules for constructing the tableau calculus of HdPL; any rule premise must not be an accessibility formula when we apply $[\dia_2^d]$, $[Id_1]$, and $[Id_2]$.

\begin{rem}
  Consider, for example, the following branch (here $a, b \in \Nom_2$):
  \[
    \infer[{[\diamondsuit_2^d]}]
      {\deduce{@_i @_c \diamondsuit_2 b}{@_i @_a \diamondsuit_2 c}}
      {@_i @_a \diamondsuit_2 \diamondsuit_2 b} .
  \]
  Then, $@_i @_a \diamondsuit_2 c$ is an accessibility formula, but $@_i @_c \diamondsuit_2 b$ is not. This is because $b$ is not a new nominal applying $[\diamondsuit_2^d]$.
\end{rem}

In summary, the tableau calculus for HdPL is obtained by modifying the tableau calculus for HPL in the following way:
\begin{enumerate}[i)]
  \item Replacing $[\dia_2]$ and $[\neg \dia_2]$ with $[\dia^d_2]$ and $[\neg \dia^d_2]$, respectively.
  \item Adding a new rule when applying $[\dia_2^d]$, $[Id_1]$, and $[Id_2]$: none of the premises is an accessibility formula.
\end{enumerate}

\subsection{Soundness and Completeness}
\begin{thm}[soundness]
  The tableau calculus for HPL is sound for the class of product frames, i.e., if $\vdash \varphi$, then $\models \varphi$.
  \label{thmsoundd}
\end{thm}
\begin{proof}
  The only differences between this proof and that of Theorem \ref{thmsound} are the cases of $[\dia_2^d]$, $[\neg \dia_2^d]$, $[Id_1]$, and $[Id_2]$ in the proof of Lemma \ref{lemsound}. We only do the case of $[\dia_2^d]$. The other cases are left to the reader.
  
  Suppose we acquire two new formulae $@_j @_b \diamondsuit_2 c$ and $@_j @_c \psi$ by applying $[\dia_2^d]$ to $\Theta^m$. Let $\Theta^{m+1}$ be a branch that contains both formulae. Since we can apply $[\dia_2^d]$ to $\Theta^m$, $@_j @_b \diamondsuit_2 \psi \in \Theta^m$ holds, and by assumption $\mathfrak{M}, (j^V, b^V) \models \diamondsuit_2 \psi$ holds. This implies that there is some $y' \in W_2$ such that $b^V R_2(j^V) y'$ and $\mathfrak{M}, (j^V, y') \models \psi$. Now, we define a new valuation function $V'$:
    \[
     V'(s) = 
     \begin{cases}
       W_1 \times \{ y' \} & (s = c) \\
       V(s) & (\text{otherwise}) .
     \end{cases} 
    \]
    Define $\mathfrak{M}' = (\mathfrak{F}, V')$. Then, we have $b^{V'} R_2(j^{V'}) c^{V'}$ and $\mathfrak{M}', (j^{V'}, c^{V'}) \models \psi$. By the latter,  Definition \ref{deffaithful} (\ref{sound1}) holds for $@_j @_c \psi \in \Theta^{m+1}$ and $(j^{V'}, c^{V'})$. By the former, $\mathfrak{M}', (j^{V'}, b^{V'}) \models \diamondsuit_2 c$. Then, Definition \ref{deffaithful} (\ref{sound1}) holds for $@_j @_b \diamondsuit_2 c \in \Theta^{m+1}$. Moreover, since $c$ does not occur in $\Theta^m$, the interpretation of any formula in $\Theta^m$ does not change even if $V$ is changed into $V'$. Therefore, $\mathfrak{M}'$ is faithful to $\Theta^{m+1}$.
\end{proof}

To show completeness, we first show that the quasi-subformula property holds even in the tableau calculus for HdPL.

\begin{lem}
  Let $\mathcal{T}$ be a tableau. If a formula in $\mathcal{T}$ has the form $@_i @_a \varphi$, then one of the following statements holds:
  \begin{itemize}
    \item $@_i @_a \varphi$ is a quasi-subformula of the root formula of $\mathcal{T}$.
    \item $@_i @_a \varphi$ is an accessibility formula.
  \end{itemize}
  \label{lemquasid}
\end{lem}
\begin{proof}
  We only do the case of $[\dia_2^d]$. The other cases are left to the reader. Note that we can prove this lemma in the same way as Lemma \ref{lemquasi}, except the cases of $[\dia_2^d]$, $[\neg \dia_2^d]$, $[Id_1]$, and $[Id_2]$.
  
  Let us consider the case of $[\dia_2^d]$. First, $@_i @_a \dia_2 b$ is an accessibility formula. Also, $@_i @_b \varphi$ is a quasi-subformula of $@_i @_a \diamondsuit_2 \varphi$, and by the induction hypothesis, $@_i @_a \diamondsuit_2 \varphi$ is a quasi-subformula of the root formula.
\end{proof}

Next, we change the definition of saturation (Definition \ref{defsat}) of the HdPL tableau calculus as follows. 

\begin{dfn}
  A branch $\Theta$ of an HdPL tableau is saturated if $\Theta$ satisfies all the conditions below (we show only the differences from Definition \ref{defsat}):
  \begin{description}
  \item[\ref{satdia2})] If $@_i @_a \diamondsuit_2 \varphi \in \Theta$ and it is not an accessibility formula, then there is some $b \in \textbf{Nom}_2$ such that $@_i @_a \diamondsuit_2 b, @_i @_b \varphi \in \Theta$.
  \item[\ref{satnegdia2})] If $@_i @_a \neg \diamondsuit_2 \varphi, @_i @_a \diamondsuit_2 b \in \Theta$, then $@_i @_b \neg \varphi \in \Theta$.
  \item[\ref{satid11})] If $@_i @_a \varphi, @_i j \in \Theta$ and $@_i @_a \varphi$ is not an accessibility formula, then $@_j @_a \varphi \in \Theta$.
  \item[\ref{satid12})] If $@_i @_a \varphi, @_a b \in \Theta$ and $@_i @_a \varphi$ is not an accessibility formula, then $@_i @_b \varphi \in \Theta$.
  \end{description}
  \label{defsatd}
\end{dfn}

The discussion about right nominals does not change in proving the completeness of the tableau calculus for HdPL. Thus, we can use the same definition of an urfather of a nominal as that for HPL. Moreover, the properties of urfathers (Lemma \ref{lemurpro}) are preserved in the tableau calculus for HdPL. 

Then, we are ready to construct a d-product model from a branch of an HdPL tableau. Note that the only difference is how we define the accessibility relations $R_2$.

\begin{dfn}
  Given an open branch $\Theta$ of an HdPL tableau, a d-product model $\mathfrak{M}^{\Theta} = (W^{\Theta}, R_h^{\Theta}, R_v^\Theta, V^{\Theta})$ is defined as follows:
  \begin{align*}
    W_1^{\Theta} &= \{ u_\Theta(i) \in \Nom_1 \mid i \ \text{occurs in} \ \Theta \}, \\
    W_2^{\Theta} &= \{ u_\Theta(a) \in \Nom_2 \mid a \ \text{occurs in} \ \Theta \}, \\
    W^{\Theta} &= W_1^\Theta \times W_2^\Theta,  \\
    R_h^{\Theta} &= \{ ((u_\Theta(i), u_\Theta(a)), (u_{\Theta}(j), u_\Theta(a))) \mid @_i \diamondsuit_1 j \in \Theta \}, \\
    R_v^{\Theta}(u_\Theta(i)) &= \{ ((u_\Theta(i), u_\Theta(a)), (u_\Theta(i), u_{\Theta}(b))) \mid @_i @_a \diamondsuit_2 b \in \Theta \}, \\
    V^{\Theta}(p) &= \{ (u_\Theta(i), u_\Theta(a)) \mid @_i @_a p \in \Theta \}, \quad \text{where} \quad p \in \Prop, \\
    V^{\Theta}(i) &= \{ u_{\Theta}(i) \} \times W_2, \quad \text{where} \quad i \in \Nom_1, \\
    V^{\Theta}(a) &= W_1 \times \{ u_{\Theta}(a) \},  \quad \text{where} \quad a \in \Nom_2.
  \end{align*}
\end{dfn}

\begin{lem}[model existence theorem for HdPL]
  Let $\Theta$ be an open branch of an HdPL tableau and $@_i @_a \varphi$ be a quasi-subformula of a root formula. Then, we have
    \[
      \text{if } @_i @_a \varphi \in \Theta, \text{ then } \mathfrak{M}^{\Theta}, (u_\Theta(i), u_{\Theta}(a)) \models \varphi.
    \]
  \label{lemmodelexistd}
\end{lem}
\begin{proof}
  By induction on the complexity of $\varphi$. We only do the cases of $\varphi = \diamondsuit_2 \psi$ and $\varphi = \neg \diamondsuit_2 \psi$; the other cases are same as in the proof of Lemma \ref{lemmodelexist}.
  \begin{description}
    \item[{$[\varphi = \diamondsuit_2 \psi]$}] Suppose that $@_i @_a \diamondsuit_2 \psi \in \Theta$. Since $\Theta$ is saturated, by Definition \ref{defsatd} (\ref{satdia2}) there is a nominal $b \in \Nom_2$ such that $@_i @_a \diamondsuit_2 b, @_i @_b \psi \in \Theta$. From the former and the definition of $R_v^{\Theta}$, we have $(u_{\Theta}(i), u_{\Theta}(a)) R_v^\Theta (u_{\Theta}(i), u_\Theta(b))$, and by applying the induction hypothesis to the latter, we have $\mathfrak{M}^{\Theta}, (u_\Theta(i), u_{\Theta}(b)) \models \psi$. Therefore $\mathfrak{M}^{\Theta}, (u_\Theta(i), u_{\Theta}(a)) \models \diamondsuit_2 \psi$.

    \item[{$[\varphi = \neg \diamondsuit_2 \psi]$}] Suppose that $@_i @_a \neg \diamondsuit_2 \psi \in \Theta$. By Lemma \ref{lemurpro} (\ref{lemurcl}) we have $@_{u_\Theta(i)} @_{u_\Theta(a)} \neg \diamondsuit_2 \psi \in \Theta$. Now, assume that there is a nominal $b \in \Nom_2$ such that $(u_\Theta(i), u_\Theta(a)) R_v^\Theta (u_\Theta(i), u_\Theta(a))$ and $@_i @_a \dia_2 b \in \Theta$ (if such a $b$ does not exist, in the model $\mathfrak{M}^\Theta$ there is no possible world that is reachable from $(u_\Theta(i), u_\Theta(a))$ through $R_v^\Theta$). Thus, $\mathfrak{M}^{\Theta}, (u_\Theta(i), u_{\Theta}(a)) \models \neg \diamondsuit_2 \psi$.) If we take such a $b$, since $\Theta$ is saturated, by Definition \ref{defsatd} (\ref{satnegdia2}) we have $@_i @_b \neg \psi \in \Theta$. Hence, by the induction hypothesis, $\mathfrak{M}^{\Theta}, (u_\Theta(i), u_{\Theta}(b)) \models \neg \psi$. This holds for all $b$, so we have $\mathfrak{M}^{\Theta}, (u_\Theta(i), u_\Theta(a)) \models \square_2 \neg \psi$. Therefore, $\mathfrak{M}^{\Theta}, (u_\Theta(i), u_\Theta(a)) \models \neg \diamondsuit_2 \psi$.
  \end{description}
\end{proof}

Thus, we have reached our goal.

\begin{thm}[completeness]
  The tableau calculus for HdPL is complete for the class of d-product frames, i.e., if $\models \varphi$, then $\vdash \varphi$.
  \label{thmcompleted}
\end{thm}

\subsection{Restrictions on Frames}
\label{subsecinc}

By adding a special axiom to the HdPL, it is possible to place restrictions on the accessibility relation of the Kripke frame. We focus on the property ``frame is decreasing'' introduced in \cite{sano2010}. 

An example in which this property is used is $T \times W$ logic \cite{von1997, thomason2002_7}, which describes histories with the same time order. Let us examine the semantics of $T \times W$ logic, called a $T \times W$ frame. 

\begin{dfn}[\cite{von1997}, Definition 1.1]
  A $T \times W$ frame is a quadruple $\Kframe = (T, <, W, (\sim_t)_{t \in T})$, where
  \begin{itemize}
    \item $T$ is a non-empty set of time points, 
    \item $<$ is a linear ordering on $T$,
    \item $W$ is a non-empty set of worlds (or histories), and
    \item for each $t \in T$, $\sim_t$ is an equivalence relation on $W$
  \end{itemize}
  that satisfies the following condition:
  \begin{center}
    For all $w, w' \in W$ and $t, t' \in T$, if $w \sim_t w'$ and $t' < t$, then $w \sim_{t'} w'$.
  \end{center}
\end{dfn}

This logic illustrates that the more time passes, the more differentiated the world becomes. In other words, if we can assume that some two histories are equivalent, then the equivalence holds at any time point in the past. Thus, we need a property to express that ``every relationship that exists in the present has existed in the past.'' This statement corresponds to the property of decreasing. \footnote{\cite{sano2010} introduced a $T \times W$ frame as an example of an increasing model. If we interpret $\dia_1 p$ as ``p holds in some \emph{past} point,'' we can assume that $T \times W$ frame is increasing.}

Then, we try adding a corresponding rule of the tableau calculus.

\begin{dfn}
  A d-product frame $\mathfrak{F} = (W_1 \times W_2, R_h, R_v)$ is \emph{decreasing} if for all $x, x' \in W_1$, if $x R_1 x'$, then $R_2(x) \supseteq R_2(x')$.
\end{dfn}

\begin{prop}[\cite{sano2010}, Proposition 4.15.] The property that a d-product frame $\mathfrak{F} = (W_1 \times W_2, R_h, R_v)$ is decreasing is definable by the following formula:
  \begin{equation}
    \dia_1 @_a \diamondsuit_2 b \rightarrow @_a \diamondsuit_2 b . \tag{AxDec}
    \label{axdec}
  \end{equation}
  \label{propdec}
\end{prop}
\begin{proof}
  It suffices to show the following property:
  \begin{center}
    $\mathfrak{F}$ is decreasing $\iff$ (\ref{axdec}) is valid in $\mathfrak{F}$.
  \end{center}
  \begin{description}
    \item[($\Rightarrow$)] Suppose that $\mathfrak{F}$ is decreasing. Take a valuation function $V$ and a point $(x, y) \in W_1 \times W_2$ such that $(\mathfrak{F}, V), (x, y) \models \dia_1 @_a \diamondsuit_2 b$ (we abbreviate it as $(x, y) \models \varphi$). Then, there is an $x' \in W_1$ such that $x R x'$ and $(x', y) \models @_a \diamondsuit_2 b$, and by definition, $a^V R_2(x') b^V$ holds. Since $\mathfrak{F}$ is decreasing, we have $a^V R_2(x) b^V$. Therefore, $(x, y) \models @_a \diamondsuit_2 b$ holds.
    \item[($\Leftarrow$)] Suppose that (\ref{axdec}) is valid in $\mathfrak{F}$. Take states $x, x' \in W_1$ and $y, y' \in W_2$ such that $x R_1 x'$ and $y R_2(x') y'$. Also, take a valuation function $V$ such that $a^V = y$ and $b^V = y'$. Then, $(x', y) \models @_a \diamondsuit_2 b$ holds, and by $x R_1 x'$ we have $(x, y) \models \dia_1 @_a \diamondsuit_2 b$. Since (\ref{axdec}) holds in all worlds in this model, we acquire that $(x, y) \models \dia_1 @_a \diamondsuit_2 b$. This implies that $y R_2(x) y'$. Therefore, we conclude that $R_2(x) \supseteq R_2(x')$.
  \end{description}
\end{proof}

So far, we have confirmed that the frame property of decreasing is definable in a Hilbert-style axiomatization of HdPL. Now, we show that we can construct the HdPL tableau calculus corresponding to decreasing frames. Note that the way to construct this is to a add a rule, not an axiom.

\begin{dfn}
  The rule $[Dec]$ is defined as follows:
  \[
    \infer[{[Dec]}]
      {@_i @_a \diamondsuit_2 b}
      {\deduce{@_i \diamondsuit_1 j} {@_j @_a \diamondsuit_2 b}} .
  \]
  where the formulae both above and below the line are accessibility formulae.
  
  We write $\mathbf{TAB} + Dec$ for an HdPL tableau calculus with the $[Dec]$ rule.
  
  We have to fix the definition of saturation. A branch of a tableau of $\mathbf{TAB} + Dec$ is saturated if it satisfies not only all the conditions of Definition \ref{defsatd} but also the following condition:
  \begin{itemize}
  \item If $@_j @_a \diamondsuit_2 b, @_i \diamondsuit_1 j \in \Theta$ and they are accessibility formulae, then $@_i @_a \diamondsuit_2 b \in \Theta$.
\end{itemize}
\end{dfn}

\begin{ex}
  The axiom (\ref{axdec}) of Proposition \ref{propdec} is provable in $\mathbf{TAB} + Dec$. Figure \ref{figdectab} is a proof using tableau calculus. Note that $\neg (\dia_1 @_a \diamondsuit_2 b \rightarrow @_a \diamondsuit_2 b)$ is equivalent to the formula $\dia_1 @_a \diamondsuit_2 b \land \neg @_a \diamondsuit_2 b$.
\end{ex}

\begin{figure}
  \begin{align*}
  1. &@_{i_0} @_{a_0} (\dia_1 @_a \diamondsuit_2 b \land \neg @_a \diamondsuit_2 b) \\
  2. &@_{i_0} @_{a_0} \dia_1 @_a \diamondsuit_2 b &(1, [\land]) \\
  3. &@_{i_0} @_{a_0} \neg @_a \diamondsuit_2 b &(1, [\land]) \\
  4. &{@_{i_0} \diamondsuit_1 i_1}^* &(2, [\dia_1]) \\
  5. &@_{i_1} @_{a_0} @_a \diamondsuit_2 b &(2, [\dia_1]) \\
  6. &@_{i_0} @_a \neg \diamondsuit_2 b &(3, [\neg @_2]) \\
  7. &@_{i_1} @_a \diamondsuit_2 b &(5, [@_2]) \\
  8. &{@_{i_1} @_a \diamondsuit_2 a_1}^* &(7, [\dia_2]) \\
  9. &@_{i_1} @_{a_1} b &(7, [\dia_2]) \\
  10. &{@_{i_0} @_a \diamondsuit_2 a_1}^* &(4, 8, [Dec]) \\
  11. &@_{i_0} @_{a_1} \neg b &(6, 10, [\neg \dia_2]) \\
  12. &@_{a_1} b &(9, [Red_2]) \\
  13. &@_{a_1} \neg b &(11, [Red_2]) \\
  14. &\bot &(12, 13)
  \end{align*}
  The formulae with ${}^*$ are accessibility formulae.
  \caption{A tableau for HdPL that proves (\ref{axdec}).}
  \label{figdectab}
\end{figure}

What we do next is to show that $\mathbf{TAB} + Dec$ is sound and complete for the class of decreasing d-product frames.

\begin{thm}
  The tableau calculus $\mathbf{TAB} + Dec$ is sound for the class of decreasing d-product frames.
\end{thm}
\begin{proof}
  It suffices to show that Lemma \ref{lemsound} holds for $[Dec]$. We can assume that every d-product frame is decreasing.
  
  Let $@_j @_b \diamondsuit_2 c$ be a formula acquired by applying $[Dec]$ to $\Theta^m$. Let $\Theta^{m+1}$ be a branch obtained by adding $@_j @_b \diamondsuit_2 c$ to $\Theta^m$. Then, there is some $k \in \Nom_1$ such that $@_j \diamondsuit_1 k, @_k @_b \diamondsuit_2 c \in \Theta^m$. Thus, for all $y \in W_2$, we have $\mathfrak{M}, (j^V, y) \models \diamondsuit_1 k$, and $\mathfrak{M}, (k^V, b^V) \models \diamondsuit_2 c$ also holds. They imply $j^V R_1 k^V, b^V R_2(k^V) c^V$. Since we are discussing a decreasing model, it follows that $b^V R_2(j^V) c^V$. Hence, $\mathfrak{M}, (j^V, b^V) \models \diamondsuit_2 c$, and this shows that $@_j @_b \diamondsuit_2 c \in \Theta^{m+1}$ satisfies the condition of Definition \ref{deffaithful} (\ref{sound1}).
\end{proof}

Proving the completeness is a little more complicated. First, we show the following lemma:

\begin{lem}
  Let $\Theta$ be an open, saturated branch of a tableau of $\mathbf{TAB} + Dec$. Then, the d-product model $\mathfrak{M}^\Theta$ of Definition \ref{defmodel} is decreasing.
  \label{leminc}
\end{lem}
\begin{proof}
  Take $u_\Theta(i), u_\Theta(j) \in W_1^\Theta, u_\Theta(a)$ and $u_\Theta(b) \in W_2^\Theta$ such that $u_\Theta(i) R_1^\Theta u_\Theta(j)$ and $u_\Theta(a) R_2^\Theta(u_\Theta(j)) u_\Theta(b)$. Then, we have $@_i \diamondsuit_1 j, @_j @_a \diamondsuit_2 b \in \Theta$. We have to consider two different cases in which $@_j @_a \diamondsuit_2 b$ is or is not an accessibility.
  \begin{enumerate}
    \item Suppose that $@_j @_a \diamondsuit_2 b$ is an accessibility formula. In this case, since $\Theta$ is saturated, we have $@_i @_a \diamondsuit_2 b \in \Theta$ (by the condition corresponding to $[Dec]$). Therefore, $u_\Theta(a) R_2^\Theta(i) u_\Theta(b)$.
    \item Suppose that $@_j @_a \diamondsuit_2 b$ is not an accessibility formula. Since $\Theta$ is saturated, by Definition \ref{defsat} (\ref{satdia2}) there is some $c \in \Nom_2$ such that $@_j @_a \diamondsuit_2 c, @_j @_c b \in \Theta$. Moreover, by Definition \ref{defsat} (\ref{satred2}) we have $@_c b \in \Theta$. Since $@_j @_a \diamondsuit_2 c$ is an accessibility formula, $@_i @_a \diamondsuit_2 c \in \Theta$ holds. Hence, we have $u_\Theta(a) R_2^\Theta(u_\Theta(i)) u_\Theta(c)$. Furthermore, by $@_c b \in \Theta$ and Lemma \ref{lemurpro} (\ref{lemureq}) it follows that $u_\Theta(c) = u_\Theta(b)$. Therefore, we have $u_\Theta(a) R_2^\Theta(i) u_\Theta(b)$.
  \end{enumerate}
\end{proof}

Now we are ready to show the completeness.

\begin{thm}
  The tableau calculus $\mathbf{TAB} + Dec$ is complete for the class of decreasing d-product frames.
\end{thm}
\begin{proof}
  Note that only accessibility formulae are added when we use the $[Dec]$ rule. Then, Lemma \ref{lemmodelexist} holds even for a tableau branch of $\mathbf{TAB} + Dec$. Moreover, by Lemma \ref{leminc} $\mathfrak{M}^\Theta$ is decreasing. From these facts, the theorem follows.
\end{proof}

\section{Conclusions and Future Work}

In this paper, we construct tableau calculi for HPL and HdPL and show the soundness and completeness of them. However, there is still room for further work on the tableau calculus of many-dimensional hybrid (dependent) product logic.

\subsection{Termination and Decidability}

As mentioned in Section \ref{secintro}, the tableau calculus for HPL introduced above does not have the termination property. If a tableau calculus had the termination property, then we could show the decidability of HPL. However, unfortunately, the following proposition holds:

\begin{prop}
  There is a root formula with which the tableau for HPL has a branch of infinite length. 
\end{prop}
\begin{proof}
  Consider the formula $@_i @_a (\dia_1 p \land \square_1 \dia_2 q \land \square_2 \dia_1 r)$, where $p, q, r \in \Prop$. Let us use it as the root formula and construct a tableau. Then, an infinite branch is generated, as shown in Figure \ref{figinftableau} (looping 7--12, where $i_1, i_2, a_1$ are replaced with $i_n, i_{n+1}, a_n$, respectively). Note that we also use $[\square_1], [\square_2]$ mentioned in Remark \ref{remsquare}. 
\end{proof}

\begin{figure}[tb]
\begin{align*}
  1. &@_{i_0} @_{a_0} (\dia_1 p \land \square_1 \dia_2 q \land \square_2 \dia_1 r) \\
  2. &@_{i_0} @_{a_0} \dia_1 p &(1, [\land]) \\
  3. &@_{i_0} @_{a_0} \square_1 \dia_2 q &(1, [\land]) \\
  4. &@_{i_0} @_{a_0} \square_2 \dia_1 r &(1, [\land]) \\
  5. &@_{i_0} \dia_1 {i_1} &(2, [\dia_1]) \\
  6. &@_{i_1} @_{a_0} p &(2, [\dia_1]) \\ 
  7. &@_{i_1} @_{a_0} \dia_2 q &(3, 5, [\square_1]) \\
  8. &@_{a_0} \dia_2 {a_1} &(7, [\dia_2]) \\
  9. &@_{i_1} @_{a_1} q &(7, [\dia_2]) \\
  10. &@_{i_0} @_{a_1} \dia_1 r &(4, 8, [\square_2]) \\
  11. &@_{i_0} \dia_1 {i_2} &(10, [\dia_1]) \\
  12. &@_{i_2} @_{a_1} r &(10, [\dia_1]) \\
  13. &@_{i_2} @_{a_0} \dia_2 q &(3, 11, [\square_1]) \\
  14. &@_{a_0} \dia_2 {a_2} &(13, [\dia_2]) \\
  15. &@_{i_2} @_{a_2} q &(13, [\dia_2]) \\
  16. &@_{i_0} @_{a_2} \dia_1 r &(4, 14, [\square_2]) \\
  17. &@_{i_0} \dia_1 {i_3} &(16, [\dia_1]) \\
  18. &@_{i_3} @_{a_2} r &(16, [\dia_1]) \\
  &\vdots
\end{align*}
\caption{An example of a tableau for HPL that has an infinite branch.}
\label{figinftableau}
\end{figure}

Thus, if we want to construct a tableau calculus with termination, we have to modify the present system, such as fixing some rules or adding some restrictions.

One possible solution is loop-checking. This method constructs finite branches without creating nominals that play the same role, called \emph{twins}. In fact, in the example of Figure \ref{figinftableau}, both $(i_1, a_1)$ and $(i_2, a_2)$ are twins since they make the same propositional variable $q$ true. Thus, proving the termination and completeness of the tableau calculus for HPL with loop-checking would be successful. The mechanism is introduced in some papers; for example, see \cite{bolander2007, bolander2009, hoffmann2010}. \footnote{I thank an anonymous referee for giving me some helpful advice on this subsection.}

The problem of whether the tableau calculus for HdPL has termination is still an open problem. Note that with a formula $@_i @_a (\dia_1 p \land \square_1 \dia_2 q \land \square_2 \dia_1 r)$ we can construct a finite tableau for HdPL (we leave it to the reader).

\subsection{What if We Have More Dimensions?}

The tableau calculus we have can be easily extended to $n$-dimensional HPL ($n \geq 3$). Soundness and completeness can be proved with a method similar to that of Sections \ref{secsound} and \ref{seccomplete}. 

However, we cannot construct any tableau calculus with termination, since it is known that every product modal logic with more than 3 dimensions is undecidable (see \cite{hirsch2002, gabbay2003}).

\subsection{Adding More Operators}

The system we have considered so far is a system containing only the most basic operator, @. The other operators in hybrid logic are the existential operator $E$ and the downarrow operator $\downarrow$. In orthodox (i.e., one-dimensional) hybrid logic, $E \varphi$ can be interpreted as ``there exists a possible world where $\varphi$ is true'' and $\downarrow x.\varphi$ can be interpreted as ``when the current world is described as $x$, $\varphi$ is true.'' See \cite{blackburn2006P} for more details.

Let $E_1$ and $E_2$ be existential operators and $\downarrow_1 x$ and $\downarrow_2 y$ be downarrow operators for the first and the second dimension, respectively. The rules for these operators will be as follows (we show only the rules for $E_1$ and $\downarrow_1 x$; the other rules can be defined similarly):

\begin{gather*}
    \infer[{[E_1]}^{*1}]
      {@_j @_a \varphi}
      {@_i @_a E_1 \varphi}
    \qquad
    \infer[{[\neg E_1]}^{*2}]
      {@_j @_a \neg \varphi}
      {@_i @_a \neg E_1 \varphi}
    \\
    \infer[{[\downarrow_1]}^{*3}]
      {@_i @_a \varphi[i/x]}
      {@_i @_a \downarrow_1 x. \varphi}
    \qquad
    \infer[{[\neg \downarrow_1]}^{*3}]
      {@_i @_b \neg \varphi[i/x]}
      {@_i @_a \neg \downarrow_1 x. \varphi} \\
  \end{gather*}

  *1: $j \in \Nom_1$ does not occur in the previous part of this branch.

  *2: $j \in \Nom_1$ has to occur in this branch.
  
  *3: $\varphi[i/x]$ is a formula obtained by replacing all the occurrences of $x$ in $\varphi$ with a nominal $i$.

\vspace{\baselineskip}
It is clear that any tableau calculus for HPL (or HdPL) with $[\downarrow_1]$ and $[\downarrow_2]$ does not have the termination property, since it is shown in \cite{areces1999} that general hybrid logic containing $\downarrow$ is undecidable. Moreover, if we add both $[E_1]$ and $[E_2]$ to the tableau calculus, then we lose the termination property \cite[Theorem 5.37]{gabbay2003}. We do not know whether this is also the case for the tableau calculus for HPL with only $[E_1]$.

\subsection{Adding More Rules}

Adding axioms to the hybrid tableau calculus makes them richer. With the proper addition of rules, we can construct a tableau calculus for \textbf{S4}.

In a previous work \cite{bolander2009}, some axioms corresponding to reflexivity, irreflexivity, and transitivity were added to the tableau calculus for hybrid logic with preserving termination. Applying this method, we also acquire a many-dimensional tableau calculus corresponding to \textbf{S4}$^n$. In terms of automatic calculation, however, it will be better to add axioms in the form of an inference rule.

Moreover, for the HdPL tableau calculus, some rules corresponding to dependencies can be added as in Section \ref{subsecinc}. These should also be added in the form of inference rules. 

\section*{Acknowledgments}
This paper is based on two presentations I gave at conferences in Japan. The first one, related to Sections 2--5, was presented at a conference in RIMS held in December 2021 \cite{nishimura2022}. The other one, related to Section \ref{secHdPL}, was presented at MLG2021 in March 2022.

I would like to thank Ryo Kashima and Katsuhiko Sano for their guidance in writing this paper. In particular, Katsuhiko Sano taught me how to write Figures \ref{figframe} and \ref{figdframe}. Tatsuya Abe and Ken Shiotani provided comments on my presentation at the seminar. This work was supported by JST SPRING, Grant Number JPMJSP2106.

\bibliographystyle{plain}
\bibliography{logic}

\end{document}